\documentclass{lms}

\usepackage{url}
\usepackage{graphicx}
\usepackage{amssymb, amscd, amsmath}
\usepackage{xy}
\usepackage{bm}
\xyoption{all}
\usepackage[linktocpage]{hyperref}
\usepackage[dvipsnames]{xcolor}
\usepackage{outline}
\usepackage{epstopdf}
\usepackage{rotating}
\usepackage{mathabx}

\newtheorem{thm}{Theorem} 

\newtheorem{theorem}{Theorem}[section] 
\newtheorem{lemma}[theorem]{Lemma}
\newtheorem{proposition}[theorem]{Proposition}
\newtheorem{corollary}[theorem]{Corollary}
\newtheorem{assertion}[theorem]{Assertion}

\newtheorem{extensionlemma}[theorem]{Extension Lemma}
\newnumbered{definition}{Definition} 
\newnumbered{remark}{Remark}

\newcommand{\fig}[3]{\begin{figure}[h!] \includegraphics[height=#1pt]{#2}#3\end{figure}}

\newcommand{\figref}[1]{Figure~\ref{F:#1}}
\newcommand{\secref}[1]{Section~\ref{S:#1}}
\newcommand{\thmref}[1]{Theorem~\ref{T:#1}}
\newcommand{\lemref}[1]{Lemma~\ref{L:#1}}

\newcommand{\aref}[1]{Assertion~\ref{A:#1}}

\newcommand{\bz}{\mathbb Z}

\newcommand{\br}{\mathbb R}
\newcommand{\bc}{\mathbb C}
\newcommand{\bx}{\mathbb X}

\newcommand{\bw}{\mathbb W}
\newcommand{\bp}{\mathbb P}

\newcommand{\calc}{\mathcal C}
\newcommand{\calf}{\mathcal F}

\newcommand{\cs}{\mathbin{\#}} 
\newcommand{\fcs}[2]{\,_{#1}\hskip-2.5pt\cs_{#2}}
\newcommand{\ol}[1]{\overline{#1}}
\newcommand{\st}{\,\vert\,}

\newcommand{\lto}{\longrightarrow}

\newcommand{\toself}{\ \begin{sideways}$\circlearrowright$\end{sideways}\,}

\newcommand{\id}{1\!\!1}
\newcommand{\idsss}{\rm{id}_{\sss}}

\newcommand{\idm}{\textup{id}_{M}}
\newcommand{\idx}{\textup{id}_{X}}
\newcommand{\idmprime}{\textup{id}_{M'}}

\newcommand{\sss}{{S^2\hskip-2pt\times\hskip-2pt S^2}}
\newcommand{\sst}{S^2\widetilde{\times}S^2}
\newcommand{\cpone}{\bc\textup{P}^1}
\newcommand{\cptwo}{\bc\textup{P}^2}
\newcommand{\cptwobar}{\overline{\bc\textup{P}}\,\!^2}
\newcommand{\sh}{S_h}
\newcommand{\tsh}{T_h}
\newcommand{\ch}{C_h}
\newcommand{\chbar}{\ol C_h}
\newcommand{\cbar}{\ol C}
\newcommand{\kh}{K_h}
\newcommand{\khbar}{\ol K_h}
\newcommand{\rh}{\br_h}
\newcommand{\ah}{A_h}
\newcommand{\bh}{B_h}
\newcommand{\xtwo}{E(2)^{-\tau}} 
\newcommand{\xn}{E(n)^{-\tau}} 
\newcommand{\diff}{{\rm\bf Diff}}

\newcommand{\emb}{{\rm\bf Emb}}
\newcommand{\halfD}{D_0}
\newcommand{\xpq}{\bx_{p,q}}
\newcommand{\gfs}{g_{FS}}

\title[Stable isotopy in four dimensions]{Stable isotopy in four dimensions}

\author[Auckly, Kim, Melvin and Ruberman]{Dave Auckly, Hee Jung Kim, Paul Melvin and Daniel Ruberman}

\classno{57M25 (primary), 57Q60 (secondary)}

\extraline{Kim was supported by NRF grants 2012R1A1A2006981 and BK21 PLUS SNU Mathematical Sciences Division.  Ruberman was partially supported by NSF Grant 1105234 and NSF FRG Grant 1065827}

\begin{document}
\maketitle

\begin{abstract}
We construct infinite families of topologically isotopic but smoothly distinct knotted spheres in many simply-connected $4$-manifolds that become smoothly isotopic after stabilizing by connected summing with $\sss$, and as a consequence, analogous families of diffeomorphisms and metrics of positive scalar curvature for such $4$-manifolds.  We also construct families of smoothly distinct links, all of whose corresponding proper sublinks are smoothly isotopic, that become smoothly isotopic after stabilizing. 
\end{abstract}

\section{Introduction}\label{S:intro}

A basic principle of smooth $4$-dimensional topology due to Wall~\cite{wall:4-manifolds} is that homotopy equivalent, simply-connected $4$-manifolds become diffeomorphic after connected summing with sufficiently many copies of $\sss$.  In fact the same result holds for any pair of orientable homeomorphic
$4$-manifolds\,\cite{gompf:stable,kreck:surgery}.  The process of summing with a single $\sss$ is called \emph{stabilization}.  The number of stabilizations needed to obtain a diffeomorphism is not known in general, although it follows from the existence of exotic smooth structures (see \cite{donaldson:dolgachev,donaldson-kronheimer} and many subsequent papers) that at least one stabilization is often required.  In fact exactly one stabilization suffices for many important families of examples (see for example \cite{moishezon:sums,akbulut:fs-knot-surgery,auckly:stable}), and surprisingly, no examples have been found that require more than one (see~\cite{donaldson:orientation,fs:torsion,akbulut-mrowka-ruan} for an approach to this question).  Thus one might reasonably speculate that any pair of homeomorphic simply-connected $4$-manifolds become diffeomorphic after a single stabilization.  Indeed this has been shown to be the case for all {\em odd} (i.e.\ nonspin) pairs produced using `standard methods' (logarithmic transforms, knot surgeries, and rational blow downs) and thus also for all {\em even} pairs after an initial `blowup' (see below)~\cite{baykur-sunukjian:round}. 

This paper explores analogous stabilization questions for smoothly embedded $2$-spheres in \break $4$-manifolds, and for ambient diffeomorphisms derived from such spheres.  For example, it follows from Wall~\cite{wall:diffeomorphisms} and the work of Perron~\cite{perron:isotopy2} and Quinn~\cite{quinn:isotopy} that any pair of $2$-spheres embedded with simply-connected complements in a $4$-manifold that represent the same homology class are topologically isotopic, and they become smoothly isotopic after some number of stabilizations.  Again it is reasonable to ask for bounds on that number.  In this setting, however, no bounds have been established (before now) for any explicit examples.  We will produce infinite families of knotted spheres that require exactly one stabilization to become isotopic (henceforth `isotopic' will mean `smoothly isotopic' unless stated otherwise) and analogous pairs of links of arbitrarily many components all of whose corresponding proper sublinks are isotopic.   

With regard to diffeomorphisms of $4$-manifolds, the appropriate notion of stabilization is to take the connected sum with the identity map on $\sss$, explained in more detail below.  In \cite{quinn:isotopy}, Quinn showed that homotopic diffeomorphisms of any simply-connected $4$-manifold are stably isotopic, again raising the question of how many stabilizations are necessary.  Here we construct examples of infinite families for which exactly one stabilization is needed, and in a similar vein, examples of non-isotopic Riemannian metrics of positive scalar curvature that become isotopic after one stabilization.

We make the following definitions to capture these questions:

\begin{definition*}
Two simply-connected $4$-manifolds $X$ and  $Y$ are {\it $n$-stably equivalent} if 
$$
X\cs n\hskip.5pt\sss \ \cong \  Y\cs n\hskip.5pt\sss
$$
(where we write $n\hskip.5pt Z$ for the connected sum of $n$ copies of $Z$, and $\cong$ denotes orientation preserving diffeomorphism) and {\em strictly $n$-stably equivalent}, written $X\cong_n Y$, if they are $n$-stably but not $(n-1)$-stably equivalent.  There are analogous definitions for stabilizations with other manifolds $Z$ in place of $\sss$ (typically $Z=\pm\cptwo$ or $\sst$), with the corresponding strict equivalence denoted $X\cong_nY$ (mod\,$Z$), but to avoid confusion here we will reserve the term `stabilization' for $\cs\,\sss$.
 
Similarly, two embedded $2$-spheres $S\,,\,T \subset X$ (resp.\ diffeomorphisms $f\,,\,g: X\to X$) are  {\em $n\textit{-stably}$\,isotopic} 
if the natural embeddings $S\,,\,T \subset X\cs n\hskip.5pt\sss$ are isotopic (resp.\ the connected sums of $f$ and $g$ with the identity map $\id$ on $n\hskip.5pt\sss$ are isotopic; this connected sum is shown to be well-defined up to isotopy in section~\ref{S:stabilizing}; compare~\cite{giansiracusa:mcg}).  There is also a weaker notion of {\em $n$-stable equivalence}, requiring only the existence of a diffeomorphisms $\sigma$ and $\tau$ of $X\cs n\hskip.5pt\sss$ such that $T = \sigma(S)$ (resp.\ $g\cs\id = \tau\circ(f\cs\id)\circ\sigma$).  The notions of {\em strict $n$-stable isotopy} and {\em equivalence} are defined in the obvious way.  We note the simple fact that the notions of $n$-stable isotopy and equivalence are actually equivalence relations.
\end{definition*}

With this terminology, the principal goal of this paper is to produce {\em explicit} examples of strictly $1$-stably isotopic embeddings of spheres in simply-connected $4$-manifolds, and thereby to produce strictly $1$-stably isotopic diffeomorphisms of related manifolds. The emphasis is on the word `explicit', for the existence of such examples follows directly from the existence of $n$-stably isotopic spheres that are not smoothly isotopic.  The issue is that one has no {\em a priori} control on $n$.  We gain this control by working locally, within a blown-up cork (as explained in section 3), thereby producing explicit examples of strictly $1$-stably isotopic $2$-spheres in large families of `simple' $4$-manifolds.

Before proceeding, recall that {\em blowing up} an oriented $4$-manifold $X$ refers to the process of connected summing with $\pm\cptwo$.  The effect is to replace a point in $X$ with an embedded \break $2$-sphere $\pm\cpone$, the {\em exceptional sphere} of the blowup.  The choice of orientation on $\cptwo$ naturally affects the outcome, detected by the signature.  Using $+\cptwo$ yields the {\em positive blowup} $X^+$, and using $-\cptwo$ (also commonly written $\cptwobar$) yields its {\em negative blowup} $X^-$\,:
$$
X^+ \ = \ X\cs\cptwo \quad\textup{and}\quad X^- \ = \ X\cs\cptwobar
$$
This terminology is borrowed from the theory of complex surfaces, where one blows up a point by replacing it with the projective line -- or {\em exceptional curve} -- of complex tangent lines through that point, which in this case has self-intersection $-1$.  Thus the complex blowup corresponds to the negative topological blowup; as a result, some topologists prefer to call $X^-$ the {\em blowup} of $X$, and $X^+$ its {\em antiblowup}, but we do not do so here.  Note that if one restricts to blowups of the same sign, then there are many examples of homeomorphic $4$-manifolds that never become diffeomorphic no matter how many times they are blown up.  To achieve the analogue of Wall's stabilization result, one needs to allow blowups of both signs.  Indeed, it is well known that $X^{+-}\cong X\cs\sss$ when $X$ is odd, and so an equal number of positive and negative blowups will eventually yield diffeomorphic manifolds.   

Conversely, one can {\em blow down} an embedded $2$-sphere $S$ of self intersection $\pm1$ in a $4\textup{-manifold}$ $X$.  The result will be denoted by $X/S$ since, topologically, it is the quotient space obtained by collapsing $S$ to a point.  For example $X = X^\pm/S^\pm$, where $S^\pm$ is the exceptional sphere in $X^\pm$.

The embedded spheres (and related diffeomorphisms and metrics) constructed in this paper to illustrate $1$-stability all live in the family of {\em completely decomposable} $4$-manifolds
$$
\xpq \ := \ p\hskip.5pt\cptwo \cs  q\hskip.5pt\cptwobar
$$
obtained from $S^4$ by $p$ positive and $q$ negative blowups.


\begin{thm}\label{T:spheres} 
If $p \geq 4$ is even and $q \geq 5p$, then $\xpq$ contains an infinite family $\{S_k \st k\ge0\}$ of topologically isotopic embedded $2$-spheres that are pairwise strictly $1$-stably isotopic, i.e.\ smoothly isotopic in $\xpq\cs\sss$ but not in $\xpq$. For $q\ge9$,  $\bx_{2,q}$ contains a pair $S,T$ of topologically isotopic spheres that are strictly $1$-stably isotopic
\end{thm}

In fact all the spheres we construct have self-intersection $+1$, and their blowdowns have distinct Seiberg-Witten invariants.  Hence they are not only smoothly nonisotopic, but a fortiori {\em smoothly inequivalent}, i.e.\ no diffeomorphism of $\xpq$ carries any one to any other.  

Our examples arise naturally, illustrating a close relation between $1$-stable equivalence of manifolds modulo $\cptwo$ and inequivalence of spheres of self-intersection $+1$:  Start with a pair of $4$-manifolds $X\cong_1 Y$ (mod\,$\cptwo$), i.e. $X\not\cong Y$ but $X^+\cong Y^+$.  Let $S$ be the exceptional sphere in $X^+$, and $T$ be the image of the exceptional sphere in $Y^+$ under any diffeomorphism $Y^+\to X^+$.  Then $S$ and $T$ are inequivalent in $X^+$.  The harder task is to produce examples that are strictly $1$-stably isotopic, i.e.\ isotopic in $X^+\cs\sss$ (also see \secref{diff-isotopy} below).  We accomplish this by localizing our construction:   There are pairs of strictly $1$-stably isotopic spheres in the blowup $W^+$ of a certain contractible $4$-manifold $W$, the Akbulut-Mazur cork (see \secref{cork}) which has the needed properties to produce the infinite family of spheres in \thmref{spheres}.  In fact, $W$ is the first of an infinite family of corks with these same properties, each producing infinite families of spheres (see \secref{cork}).  To extend the result when $p=2$ to produce an {\em infinite} family of spheres, it would be sufficient to find a suitable embedding of such a cork in $\bx_{1,9}$.

\thmref{spheres} extends to a stabilization result for {\em links} of spheres in $4$-manifolds. The links we find come in families that satisfy  a property reminiscent of Brunnian links~\cite{brunn,debrunner:brunnian} in that, once suitably ordered, all their corresponding {\em proper} sublinks are smoothly isotopic.  We will refer to any such family of links as a {\em Brunnian family}. 


\begin{thm}\label{T:links}
Fix $m>1$.  If $p\ge2$ is even and $q\ge5p+2$ $($or $q\ge11$ when $p=2)$ then $\bx_{p+m-1,q}$ contains a Brunnian pair of topologically isotopic $m$-component ordered links that are strictly $1$-stably isotopic.  Furthermore, for $p \geq 4$ there exists an infinite family of smoothly distinct such $m$-component links.
\end{thm}

The spheres described in \thmref{spheres} in turn lead to examples of strictly $1$-stably isotopic diffeomorphisms by invoking a construction that dates back to early work of Wall~\cite{wall:diffeomorphisms}, or even earlier in Picard-Lefschetz theory~\cite{lefschetz:alg-geom}.  The key observation is that an embedded sphere $\Sigma$ of self-intersection $\pm 1$ or $\pm 2$ in a $4$-manifold $X$ gives rise to a diffeomorphism of $X$, supported in a neighborhood of $\Sigma$, that induces a `reflection' in the second homology.  An isotopy of such $2$-spheres yields an isotopy of the corresponding diffeomorphisms, by the isotopy extension theorem.  Note that the converse of this statement is not known to be true in general, but it holds `morally': the  non-isotopic diffeomorphisms constructed in~\cite{ruberman:isotopy,ruberman:polyisotopy,ruberman:swpos} are compositions of such reflections in spheres, and in the end are detected by gauge-theoretic invariants that also show these spheres to be non-isotopic. Combining these results with Theorem~\ref{T:spheres} gives $1$-stability of an infinite family of non-isotopic diffeomorphisms.


\begin{thm}\label{T:diffs}
If $p \geq 4$ is even and $q \geq 5p+2$, then $\xpq$ supports an infinite family \break $\{f_k:\xpq\toself \st k\geq 0\}$ of topologically isotopic self-diffeomorphisms that are pairwise strictly $1$-stably isotopic, i.e.\ not smoothly isotopic, but such that the diffeomorphisms $f_k \cs{\idsss}$  of $\xpq \cs \sss$ are all smoothly isotopic.  
\end{thm}
We will show that the stabilization $f \mapsto f \cs{\idsss}$ is in fact well-defined on isotopy classes of (orientation preserving) diffeomorphisms, and hence induces a homomorphism $\Phi: \pi_0(\diff(M)) \to \pi_0(\diff(M \cs \sss))$.  From this, we will deduce that for $\xpq$ as in \thmref{diffs}, the kernel of $\Phi$ is infinite.

\begin{remark*}
Versions of theorems A and C are likely to hold for other families of simply-connected closed manifolds.  For $\bx_{p,q}$ when $p$ is odd, this would require the use of gauge theory invariants for manifolds with even $b_2^+$, and in general would require a better understanding of the exotic smooth structures on other manifolds.   
\end{remark*}

The fact that the diffeomorphisms $f_k$ in Theorem~\ref{T:diffs} are not isotopic is detected by parameterized versions of the Donaldson and Seiberg-Witten invariants~\cite{ruberman:isotopy,ruberman:swpos} that work for even $p\ge4$ (extending such invariants to the $p=2$ case is analogous to defining gauge theory invariants when $b_2^+ = 1$).   Now the manifolds $\xpq$ admit metrics of positive scalar curvature (PSC for short) constructed as the connected sum~\cite{gromov-lawson:psc,schoen-yau:psc} of standard metrics $\gfs$ on $\cptwo$ and $\cptwobar$.  The $1$-parameter version of the Seiberg-Witten invariants reveal an additional fact about such metrics $g$, namely that the pull-back metrics $f_k^*g$, all of which have PSC, are not isotopic to one another, where an isotopy is a $1$-parameter family of PSC metrics.  Connected sum with a PSC metric on $\sss$ then leads to a $1$-stability result for such metrics.


\begin{thm}\label{T:psc}
If $p \geq 4$ is even and $q \geq  5p+2$, then $\xpq$ admits an infinite family \break 
$\{g_k \st k\ge0\}$ of PSC metrics that are mutually non-isotopic.  The metrics $g_k$ become isotopic through PSC metrics after connected summing with the standard PSC metric on $\sss$.
\end{thm}

\begin{remark*}
In contrast to the well-known stabilization theorems for simply-connected $4$-manifolds and diffeomorphisms, there does not seem to be a general stabilization result for PSC metrics.  The paper~\cite{botvinnik:psc-isotopy} (cf.\ also~\cite{walsh:psc-morse-I,walsh:psc-morse-II}) would suggest that such a result ought to hold, but at present there is a gap in one of the main geometric steps in~\cite{botvinnik:psc-isotopy}.  Note that the stabilization results for manifolds (under $\cs\,\sss$) generalize to  arbitrary closed oriented $4$-manifolds~\cite{gompf:stable,kreck:surgery}, but fail for nonorientable $4$-manifolds~\cite{cappell-shaneson:rp4}.  Similarly, orientability may be essential for analogous results for PSC metrics.  In particular, it seems likely that the metrics constructed in~\cite[\S 6]{ruberman:swpos} remain non-isotopic, even up to diffeomorphism, after an arbitrary number of stabilizations.
\end{remark*}

\subsection*{\bf Stable diffeomorphism versus stable isotopy}\label{S:diff-isotopy}
It is important to keep track of the distinction between equivalence of surfaces up to {\it diffeomorphism} and up to {\it isotopy}, especially in applications to questions of stable isotopy of diffeomorphisms and PSC metrics.

For knots in $S^3$, Cerf's theorem~\cite{cerf:diffS3} implies that the existence of an orientation preserving diffeomorphism taking one knot to another yields an ambient isotopy accomplishing the same thing. The analogous result (replacing {\em orientation preserving} with {\em homotopic to the identity}) is not known to be true for surfaces in $4$-manifolds. In fact we suspect that this is false.  Thus, any argument that does not produce an explicit isotopy at some point is unlikely to be a valid argument showing that two surfaces in a 4-manifold are smoothly isotopic.  Our proof of \thmref{spheres} produces such an isotopy. 

In contrast, it is easy to construct pairs of spheres that are smoothly distinct, yet become equivalent by an ambient diffeomorphism after one stabilization.  Indeed, take any two distinct, simply-connected, smooth 4-manifolds $X$ and $Y$ that are homeomorphic, but become diffeomorphic after summing with $\cptwo$, and also 
with $\sss$.  Then by choosing a diffeomorphism $f: Y\cs\cptwo \to X\cs\cptwo$ appropriately, one can assume that the spheres $\cpone$ and $f(\cpone)$ in $X
\cs\cptwo$ are topologically isotopic.  These spheres are smoothly distinct in $X \cs \cptwo$, since a diffeomorphism taking one to the other would produce a diffeomorphism between $X$ and $Y$, by blowing down.  Now for any diffeomorphism $g: (\sss)\cs X \to (\sss)\cs Y$, the diffeomorphism $$(1 \cs f) \circ (g \cs 1):(\sss) \cs X \cs \cptwo \ \lto \ (\sss) \cs X \cs \cptwo$$ takes $\cpone$ to $f(\cpone)$.  There are a few subtle points here, since the connected  sum of diffeomorphisms is not generally well defined, but it is (as shown in \secref{stabilizing}) in the cases above.


The above argument does not say anything about the existence of an isotopy between $\cpone$ and $f(\cpone)$ in $(\sss) \cs X \cs \cptwo$.  We can arrange that $(1 \cs f) \circ (g \cs 1)$ is homotopic to the identity, but it's certainly not true in general that homotopy implies isotopy~\cite{ruberman:isotopy}.  The point of our work is that by looking more carefully one
 can actually see the isotopy.

\begin{acknowledgements}\label{ackref}
The authors discussed this material at conferences held at MSRI and BIRS; we thank both institutes for their hospitality. We also thank Ryan Budney and Boris Botvinnik for helpful comments on the material related to diffeomorphisms and positive scalar curvature metrics.  A key idea in our proof of \thmref{spheres} is to perform the stabilization inside a cork.  This localization leads to readily drawn pictures of non-isotopic spheres that from other points of view would seem rather complicated and hard to work with. We are happy to acknowledge the influence of Selman Akbulut's cork-twisting technique on our approach to this problem, and to thank him for some interesting exchanges on the subject (cf.~\cite{akbulut:spheres}).
\end{acknowledgements}



\section{Some basic constructions}\label{S:prelim}


In this section, we discuss some constructions that will be used repeatedly in the proofs of our main results.

\subsection*{\bf Surfaces and handles}\label{S:surfaces}

Let $Z$ be an oriented, compact connected $4$-manifold with a fixed handle decomposition.  We view the $1$-skeleton of $Z$ as $B^4-N$, where $N$ is a union of unknotted, embedded open $2$-handles in $B^4$ attached along an unlink $L$ in $S^3$.  In pictures, following a long-standing convention, we put dots on the components of $L$ to distinguish them from the attaching circles of the $2$-handles (see \cite{kirby:4-manifolds} or~\cite{gompf-stipsicz:book} for detailed expositions of $4$-dimensional handle calculus).  As an additional notational convention, we will write $h_c$ for the $2$-handle attached along a knot $c$ in $S^3-L$.

An oriented surface $F$ embedded in $Z$ is described easily in terms of the given handle decomposition (this is discussed in~\cite{akbulut-kirby:branch} and~\cite[\S 6.2]{gompf-stipsicz:book} in some special cases):   By transversality, we can assume after an isotopy that $F$ is disjoint from the $3$ and $4$-handles, intersects the $2\textup{-handles}$ in some number of parallel copies of their cores, and intersects the $1$-handles in tubes and bands that run parallel to their cores.  The rest of $F$ can be moved to a proper critical level embedding in $B^4 - N$, with minima at radius $1/4$, saddles at radius $1/2$, and relative maxima at radius $3/4$ (although maxima will not arise for surfaces considered here).  By projecting back up to radius $1$, we see most of this surface immersed in $S^3-L$ with boundary consisting of framed pushoffs of the attaching circles of the $2$-handles and zero-pushoffs of the dotted circles, which bound disks missing from our view.  In pictures, rather than actually drawing the pushoffs of these circles, we will extend the surface to abut them by adding a boundary collar.  The $1$-handles associated with the saddles, which we refer to as {\em ribbons}, will typically be shaded darker (green) to remind us that they lie above the lighter (blue) regions around the minima.

In the arguments below, we will need to carry $F$ along as the handle decomposition of $Z$ changes by handle slides.  The simplest non-trivial case occurs when a $2$-handle $h_a$ disjoint from $F$ slides over a $2$-handle $h_c$ that meets $F$, resulting in a new $2$-handle $h_{a'}$.  This is illustrated in \figref{slide1}.  The handle $h_c$ and $F$ are unchanged; the apparent intersection of $a'$ with $F$ does not occur since $F$ is actually pushed into the $4$-ball at that point.

\fig{60}{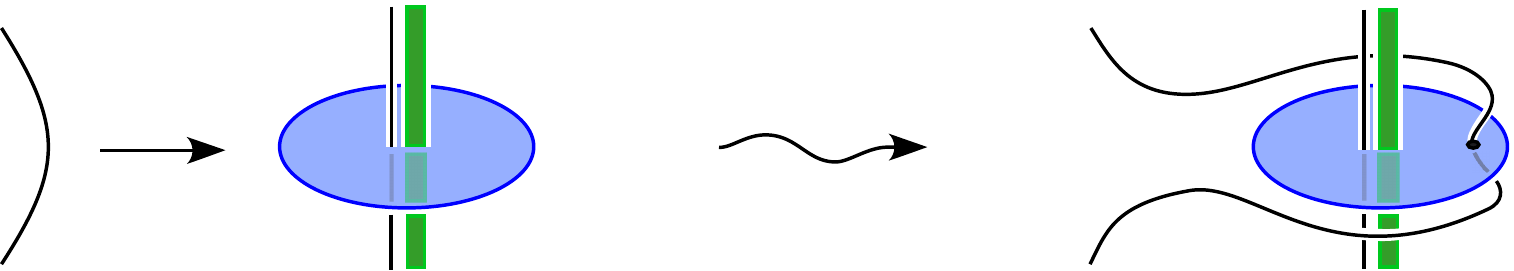}{
  \put(-328,3){\footnotesize$a$}
  \put(-225,10){\footnotesize$c$}
  \put(-90,3){\footnotesize$a'$}
  \put(-64,24){\footnotesize$c$}
  \put(-227,41){\small$1$}
  \put(-163,12){\footnotesize slide}
\caption{Sliding a $2$-handle over a $2$-handle containing a surface}
\label{F:slide1}}

By contrast, if $h_a$ meets $F$ and $h_c$ does not, then sliding $h_a$ over $h_c$ will drag a portion of $F$ over $h_c$, as indicated in \figref{slide2}.  We use $F_c$ (read `$F$ over $c$') to denote this new picture of $F$ in the new handlebody, and refer to this as a {\em surface-slide}.

\fig{60}{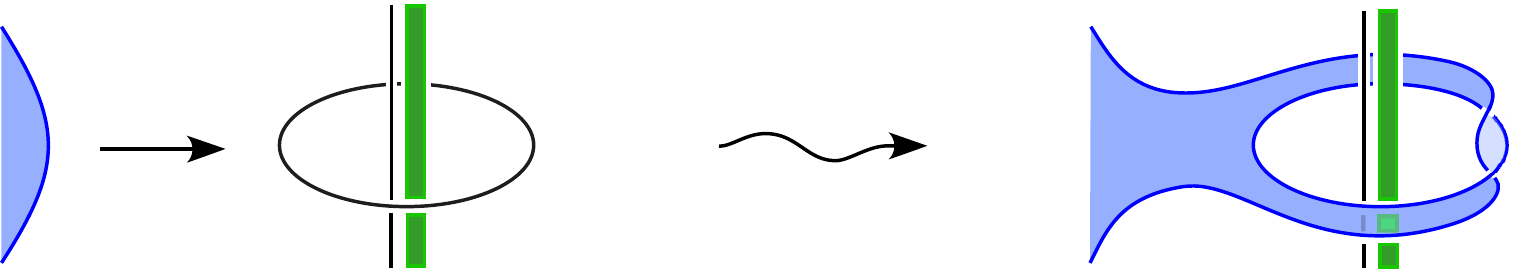}{
  \put(-326,3){\footnotesize$a$}
  \put(-222,12){\footnotesize$c$}
  \put(-89,3){\footnotesize$a'$}
  \put(-53,25){\footnotesize$c$}
  \put(-224,41){\small$1$}
  \put(-340,25){\footnotesize$F$}
  \put(-105,26){\footnotesize$F_c$}
  \put(-178,13){\footnotesize surface-slide}
\caption{Sliding a $2$-handle containing a surface over a $2$-handle}
\label{F:slide2}}

A second isotopy of this nature, called a {\em band-slide}, moves a ribbon portion of $F$ over a handle $h_c$ by the sliding the handles $h_a$ and $h_b$ that `bound' the ribbon over $h_c$, as shown in \figref{slideband}.  
We denote the resulting surface by $F_{c\ol c}$, where the $\ol c$ denotes $c$ with its orientation reversed.  This notation records the fact that the slides are oppositely oriented, since $a$ and $b$ are oppositely oriented as boundaries of $F$.  Note that $a$ and $b$ may actually coincide, in which case the handle $h_a=h_b$ is slid over $h_c$ twice with opposite orientations.

\fig{55}{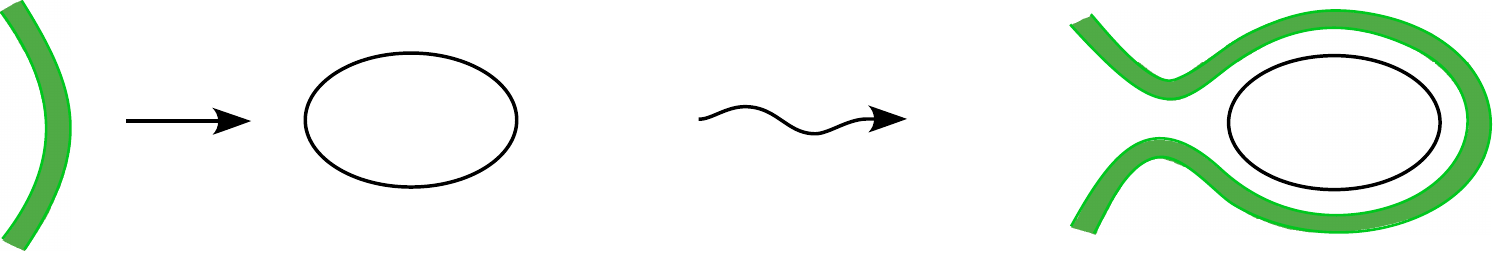}{
  \put(-327,10){\footnotesize$a$}
  \put(-312,3){\footnotesize$b$}
  \put(-328,55){\footnotesize$F$}
  \put(-217,13){\footnotesize$c$}
  \put(-98,9){\footnotesize$a'$}
  \put(-98,53){\footnotesize$F_{c\ol c}$}
  \put(-82,3){\footnotesize$b'$}
  \put(-65,28){\footnotesize$c$}
  \put(-215,40){\small$0$}
  \put(-172,13){\footnotesize band-slide}
\caption{Band-slide over a $2$-handle}
\label{F:slideband}}

For our present purposes, we need only consider surface slides of $F$ over $h_c$ when $c$ is {\em zero-framed} and {\em unknotted}.  A small isotopy of the $c$ then produces a tube in $F_c$, as shown on the right side of \figref{tubes}.  \newpage  

\fig{180}{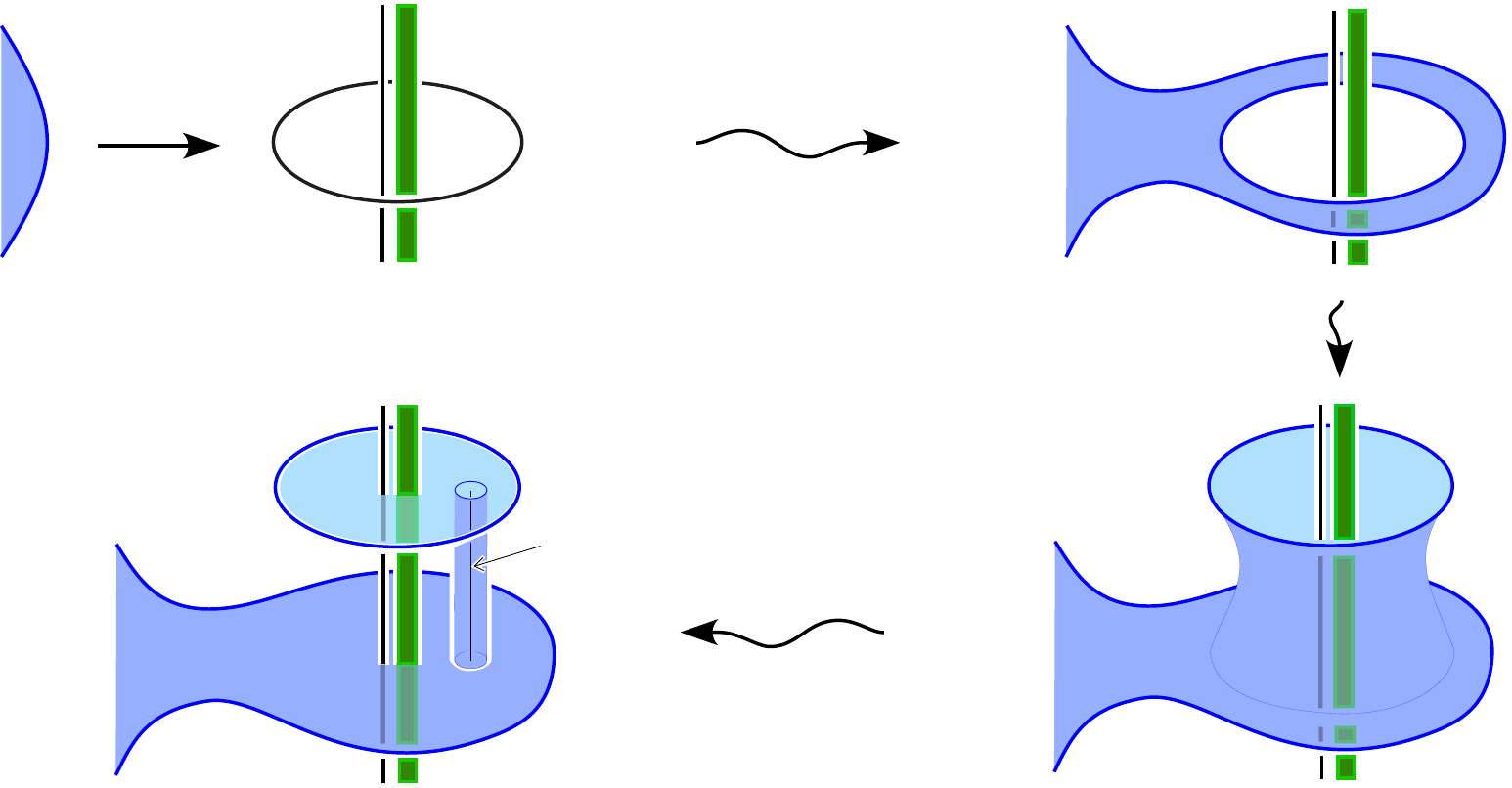}{
  \put(-338,125){\footnotesize$a$}
  \put(-264,169){\footnotesize$b$}
  \put(-235,132){\footnotesize$c$}
  \put(-96,123){\footnotesize$a'$}
  \put(-20,145){\footnotesize$c$}
  \put(-310,7){\footnotesize$a'$}
  \put(-254,80){\footnotesize$c$}
  \put(-98,6){\footnotesize$a'$}
  \put(-16,78){\footnotesize$c$}
  \put(-233,160){\small$0$}
  \put(-355,145){\footnotesize$F$}
  \put(-115,145){\footnotesize$F_c$}
  \put(-118,29){\footnotesize$F_c$}
  \put(-190,155){\footnotesize surface-slide}
  \put(-32,100){\footnotesize isotopy}
  \put(-180,22){\footnotesize isotopy}
  \put(-380,31){\footnotesize$F_c = F^c \cs_\sigma S$}
  \put(-220,54){\footnotesize$\sigma$}
\caption{Creation of tubes}
\label{F:tubes}}

Let $c'$ denote the mid height circle of this tube, which bounds a disk $D'$ in $F_c$ parallel to the core of $h_c$.  Now $c'$ may or may not bound a compatibly framed disk $D$ in $B^4-F_c$; in the case pictured it does. If $D$ exists, one can perform ambient surgery on $F_c$ to produce a new surface $F^c$ (read `$F$ under $c$') by replacing $D'$ with $D$, and represent $F_c$ as an embedded connected sum $F^c\cs_\sigma S$, where $S$ is (isotopic to) the sphere $D\cup D'$.  Here $\sigma$ is an arc that guides the tube, which can of course be moved into different positions as in the passage from the lower right to the lower left in \figref{tubes}.

Recall that the connected sum $F\cs_\sigma S$ can be defined for {\em any} two disjoint, oriented surfaces $F$ and $S$ in a $4$-manifold $Z$, where $\sigma$ is an embedded arc joining $F$ to $S$ with interior in $Z-(F\cup S)$.  This construction is independent up to isotopy of the splitting of the normal bundle of $\sigma$ one uses to specify the tube, but depends in general on the relative homotopy class of $\sigma$ (the endpoints of $\sigma$ are allowed to move in $F$ and $S$ during the homotopy).  If there is only one such homotopy class, then we say that the connected sum, simply written $F\cs S$, is {\em well-defined}.  The following two assertions are readily verified:

\begin{assertion}\label{A:welldefined}
If $Z-(F\cup S)$ is simply-connected, then $F\cs S$ is well-defined.
\end{assertion}

\begin{assertion}\label{A:cancel}
If $S$ is a sphere with trivial normal bundle and $\ol S$ is an oppositely oriented pushoff of $S$, then $(F\cs S)\cs \ol S$ is isotopic to $F$ provided both $F\cs S$ and $(F\cs S)\cs\ol S$ are well-defined.
\end{assertion}

\subsection*{\bf The key stable isotopy}\label{S:key}
Assertions \ref{A:welldefined} and \ref{A:cancel} will be used in a key situation that arises in our proof of the stabilization results, showing (under suitable conditions) that the surface obtained by sliding twice with opposite orientations over a $2$-handle  can be pulled off the handle.  Let $X$ be the $4$-ball with a $2$-handle attached along a $+1$-framed slice knot $K$, and let $S$ be the $2$-sphere of self-intersection $+1$ consisting of the core of the $2$-handle together with a chosen slice disk for $K$.  We now stabilize $X$, viewing $S\subset Z := X\cs\sss$.  This changes the handlebody by adding two $2$-handles, attached along a zero-framed Hopf link $(a,b)$.  The curves $a$ and $b$ bound disks in $B^4-S$ which, when capped off with the cores of $h_a$ and $h_b$, provide a pair of dual $2$-spheres $A$ and $B$ of self-intersection zero in $Z-S$.  

Now suppose that $S$ appears as in the top left corner of the \figref{key} with two oppositely oriented parallel sheets to one side of $(a,b)$, and a ribbon to the other.    

\fig{260}{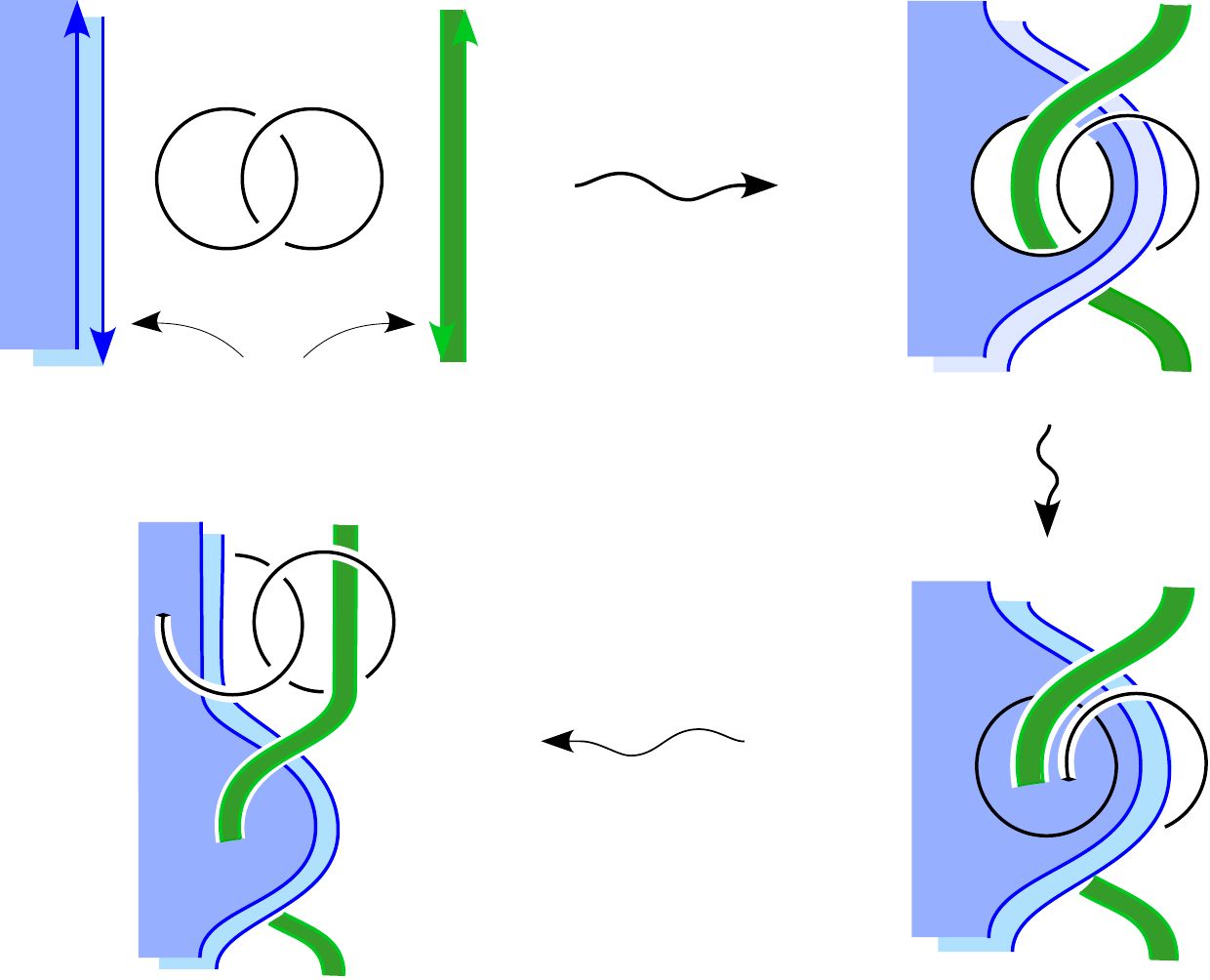}{
  \put(-278,192){\footnotesize$a$}
  \put(-225,192){\footnotesize$b$}
  \put(-6,192){\footnotesize$b$}
  \put(-65,192){\footnotesize$a$}
  \put(-63,38){\footnotesize$a$}
  \put(-5,36){\footnotesize$b$}
  \put(-280,230){\small$0$}
  \put(-292,244){\small$K$}
  \put(-225,230){\small$0$}
  \put(-253,160){\small$S$}
  \put(-41,165){\small$S_{a\ol a}$}
  \put(-255,2){\small$T$}
  \put(-40,8){\small$T$}
  \put(-259,116){\footnotesize$b$}
  \put(-242,116){\footnotesize$a$}
  \put(-167,197){\footnotesize surface and}
  \put(-166,187){\footnotesize band-slides}
  \put(-28,135){\footnotesize isotopy}
  \put(-32,125){\footnotesize assertions}
  \put(-162,50){\footnotesize isotopy}
\caption{The key stable isotopy: $S_{a\ol a}$ is isotopic to $T = S^{a\ol a}$ in $X\cs\sss$}
\label{F:key}}

Sliding both sheets over $h_a$ and the ribbon over $h_b$ moves $S$ to the sphere $S_{a\ol a}$ shown in the top right corner of the figure.  As explained above, $S_{a\ol a}$ decomposes as a connected sum of the sphere $T := S^{a\ol a}$ in $X$ shown in the bottom right corner (obtained by sliding $S$ under $a$ twice with opposite orientations) with two oppositely oriented copies of the sphere $A$, that is,
$$
S_{a\ol a} \ = \ (T\cs_\sigma A) \cs_\tau\bar A
$$
for suitable arcs $\sigma$ and $\tau$.  In fact these sums are well-defined, independent of $\sigma$ and $\tau$, by \aref{welldefined}.  
To see this for $T\cs_\sigma A$, it suffices to show that the fundamental group of the complement $C = Z-(T\cup A)$ is trivial.  This group is normally generated by the meridians of $T$ and $A$, since $Z$ is simply-connected.  But the meridian of $A$ is isotopic in $C$, via the dual sphere $B$, to a product of (a pair of oppositely oriented) meridians of $T$, and so it suffices to show that the meridian $m$ of $T$ is null-homotopic in $C$.   But since $T\cdot T = 1$, the normal bundle $N$ of $T$ is a Hopf bundle, and so $m$ is in fact null-homotopic in $\partial N$, and so also in $C$.  The same proof shows $(T\cs A) \cs_\tau\bar A$ is well-defined, since $T\cs_\sigma A$ has self-intersection $1$.  

\aref{cancel} now shows that $S_{a\ol a}$ is isotopic to $T$, which is disjoint from $h_a$ (and $h_b$).  A final isotopy of the Hopf link, twisting a half turn in a right-handed upward direction, gives an alternative view of $S^{a\ol a}$ in the bottom left picture.


\subsection*{\bf Extending a diffeomorphism from the boundary of $4$-manifold}\label{S:extend}

A crucial point for us will be to show that a diffeomorphism of the boundary of a particular $4$-manifold (a blown-up cork) extends over the $4$-manifold. The idea behind finding such extensions has been known to experts~\cite{akbulut:homology,melvin:thesis} since the mid-1970s and used many times.  For the reader's convenience, we give a clear statement of the general principle; compare~\cite[\S 2.6]{akbulut:book}. To this end, define a {\it $k$-handlebody} $X$ to be a manifold with a handle decomposition with handles of index $\le k$, and write $X^{(j)}$ for the union of handles of index $\le j$.

If $H$ is a handle of index $j$, then the boundary of its cocore is a sphere, called the {\em belt sphere} of $H$, that acquires a natural framing as a submanifold of $X^{(j-1)}$. In the case at hand, where $X$ has dimension $4$ and $j=2$, this framed sphere is just the meridian of the attaching circle of the handle, and in standard pictures of $X$ would have framing $0$.

\begin{extensionlemma}\label{L:extend}
Let $X$ be a $4$-dimensional $2$-handlebody and $f:\partial X\to\partial X$ be a diffeomorphism of its boundary.  Choose a set $L$ of meridians for the $2$-handles, framed as described above.  Then $f$ extends to a diffeomorphism of $X$ if and only if the framed link $f(L) \subset \partial X$ bounds a disjoint collection of framed disks in $X$ whose tubular neighborhood has complement diffeomorphic to $X^{(1)}$.
\end{extensionlemma}

\begin{proof}
View $X$ upside down as gotten by adding $2$, $3$, and $4$ handles to $\partial X$. Use the framings to extend over the dual $2$-handles of $X$ rel $\partial X$. Then use Laudenbach-Po\'enaru~\cite{laudenbach-poenaru:handlebodies} to extend over the rest of $X$.
\end{proof}

\subsection*{\bf Stabilizing elliptic surfaces}\label{S:elliptic}
We use~\cite{gompf-stipsicz:book} as a convenient reference for complex surfaces, and in particular for the minimal elliptic surfaces $E(n)$, which give rise to many of our examples.  Recall from \cite[7.1.11]{gompf-stipsicz:book} that if $X$ and $X'$ are $4$-manifolds containing embedded tori $T$ and $T'$ with trivial normal bundles, then we can form the {\em fiber sum} 
$$
X\fcs{T}{T'}X' \ = \ (X-N)\sqcup_\phi(X'-N')
$$
where $N$ and $N'$ are tubular neighborhoods of $T$ and $T'$, and $\phi$ is a bundle map between their boundaries.  In general this construction depends on $\phi$ (see e.g.\ \cite[Example 3.2]{gompf:symplectic}), but not in the situations we will encounter.  

We will use the fact that $E(n)$ is built as a fiber sum of $n$ copies of the rational elliptic surface $E(1)\cong\bx_{1,9}$. It follows easily that 
$$
E(n)^{+-} \ \cong \ \bx_{2n,10n} \quad \textup{ (in fact it is known that $E(n)^+\cong\bx_{2n,10n-1}$) }
$$
using the {\em Mandelbaum-Moishezon trick}~\cite{mandelbaum:ACD,mandelbaum:irrational,moishezon:sums}, which states that if $X$, $X'$ and $X\fcs{T}{T'}X'$ are simply-connected with $X$ odd, then $(X \fcs{T}{T'} X') \cs \sss \cong (X \cs X') \cs 2\,\sss$.  The simplest proof of this
uses $5$-dimensional handlebody theory: $X \fcs{T}{T'} X'$ is the upper boundary of $((X\sqcup X')\times I)\cup h \,$ where $h$ is a `toral' $1$-handle, or equivalently a $1$-handle, two $2$-handles and a $3$-handle (cf.\ `round' handle theory~\cite{asimov:round}).  Adding the $1$-handle gives $X\cs X'$ on the boundary, and then each $2$-handle contributes an $\sss$ factor (since $X\cs X'$ is simply-connected and odd). 
Now turning the $3$-handle upside down so that it becomes a $2$-handle attached to the fiber sum gives the result (cf.\ \cite{auckly:stable}).   


\section{Corks in blown-up elliptic surfaces}\label{S:cork}

The work of Freedman~\cite{freedman:simply-connected} and Donaldson~\cite{donaldson:dolgachev,donaldson-kronheimer} led to the discovery of many families of exotic smooth structures on simply-connected $4$-manifolds.  Current methods for constructing these include logarithmic transforms, knot surgeries, and rational blowdowns~\cite{friedman-morgan:complex-surfaces-I,fs:knots,fs:rationalblowdown}.

We will make use of a remarkable localization of the change in smooth structure to a contractible piece, known as a cork twist.  By definition, a {\em cork} is a pair $(C,\tau)$ where $C$ is an oriented compact contractible $4$-manifold and $\tau$ is an involution on the boundary $\partial C$ that does not extend to a diffeomorphism of the full manifold (it always extends to a homeomorphism by Freedman~\cite{freedman:simply-connected}).  Given an embedding $C\hookrightarrow X$, one can remove $C$ from $X$ and then reglue it using $\tau$.  This is the associated {\em cork twist}, denoted
$$
X^\tau \ = \ (X-C^\circ)\cup_\tau C
$$
by mild abuse of notation, as it may depend on the embedding (cf.\ Akbulut-Yasui~\cite{akbulut-yasui:knotting-corks}).  Typically $X^\tau$ will not be diffeomorphic to $X$, and if this is the case for {\em some} embedding of $C$ in $X$ then we say that $(C,\tau)$ is a {\em cork in $X$}, and that the embedding is {\em effective}.  In general, different embeddings of $C$ in $X$ are considered {\em distinct} if and only if their associated cork twists are not diffeomorphic.

    
\subsection*{\bf The Akbulut-Mazur cork $\bw$}\label{S:akbulutmazur}

That cork twisting could change the smooth structure of a manifold was discovered by Akbulut~\cite{akbulut:contractible} in 1989, and it is now known that {\em any} two homeomorphic simply-connected smooth $4$-manifolds differ by a single cork twist~\cite{curtis-freedman-hsiang-stong,matveyev:h-cobordism}.  In his construction, Akbulut used a specific Mazur manifold \cite{mazur:contractible,akbulut-kirby:mazur} \,$\bw$, depicted in \figref{Mazur} from four different perspectives:

\fig{70}{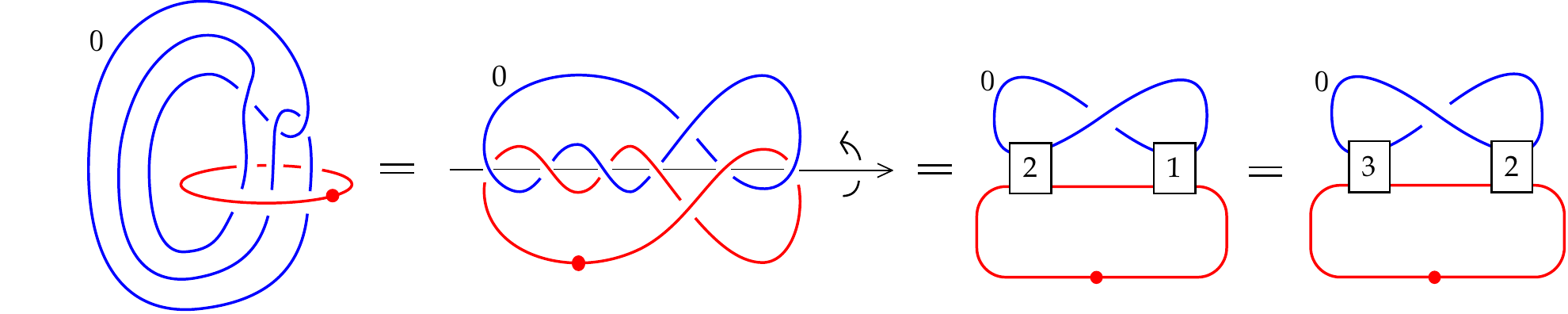}{
  \put(-365,32){$\bw \, = \, $}
  \put(-160,43){\footnotesize$\tau$}
\caption{The Akbulut-Mazur cork $(\bw,\tau)$}
\label{F:Mazur}}

\newcommand{\twistbox}[1]{{\put(8,-3){\bf ---}\put(8,2){\bf ---}}\hskip20pt{\footnotesize\boxed{#1}}{\put(0,-3){\bf ---}\put(0,2){\bf ---}}\hskip20pt}

\noindent The leftmost picture is (the mirror of) Mazur's original drawing of a pair of algebraically canceling $1$ and $2$-handles, and the others are isotopic deformations of that drawing (where the `twist box' \twistbox{k} indicates that the parallel strands entering from the left undergo $k$ full twists before exiting on the right).  Note that the boundary of $\bw$ is given by the same pictures, where the dotted circle is replaced by a $0$-framed circle.  Thus the involution $\tau$, given by a $180^\circ$ rotation about the horizontal axis in the second picture, makes sense on the boundary, where it interchanges the two surgery curves.  Since the dotted curve is the boundary of a disk in $B^4$ that has been removed, there is no obvious way to extend $\tau$ to a diffeomorphism of $\bw$, and Akbulut showed in \cite{akbulut:contractible} that in fact it does {\em not} extend.  The proof is to embed $\bw$ in a smooth $4\textup{-manifold}$ $X$ with non-trivial gauge-theoretic invariants (Akbulut used Donaldson's polynomial invariants, but now Seiberg-Witten invariants are typically used) in such a way that the corresponding invariants of $X^\tau$ vanish.  For the reader's convenience, here is more detailed sketch of this argument:

In \cite{akbulut:contractible}, Akbulut used the blown-up elliptic surface $E(2)^- = E(2)\cs\cptwobar$ as the target of his embedding of $\bw$.  Such an embedding is transparent (following the exposition in \cite{akbulut-yasui:knotting-corks}) from the rightmost drawings of $\bw$ in \figref{Mazur} and of $E(2)^-$ in \figref{E2-}.  The first drawing in \figref{E2-} is obtained from the picture of $E(2)$ in Gompf and Stipsicz's book \cite[Figure 8.16]{gompf-stipsicz:book} by a handleslide (cf.\ Figure 3 in \cite{akbulut-yasui:knotting-corks}) and a negative blowup, whose exceptional curve $e$ appears in the upper right corner.  Sliding the $+1$-framed handle $h_c$ over $h_d$ and $h_e$ yields the $-1$-framed handle $h_{c'}$ shown in the middle drawing.  Then sliding $h_a$ over $h_{c'}$ produces the $2$-handle $h_{a'}$ in the last drawing (where $a'$ is the thick figure 8) which together with the $1$-handle reproduces the last picture of $\bw$ in \figref{Mazur}.  Note that this embedding of $\bw$ is disjoint from a copy of the Gompf nucleus $N(2)$~\cite{gompf:nuclei}, given by the trefoil and its meridian in \figref{E2-}\,.

\fig{85}{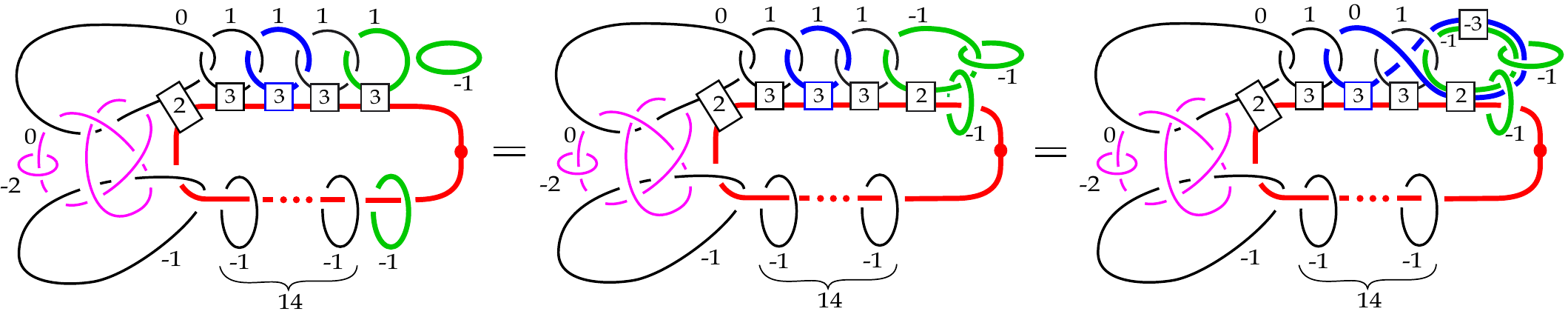}{
\put(-349,68){\footnotesize$a$}
\put(-336,68){\footnotesize$b$}
\put(-322,68){\footnotesize$c$}
\put(-319,23){\footnotesize$d$}
\put(-302,68){\footnotesize$e$}
\put(-204,68){\footnotesize$a$}
\put(-190,68){\footnotesize$b$}
\put(-175,68){\footnotesize$c'$}
\put(-16,80){\footnotesize$a'$}
\caption{Embedding the Mazur manifold $\bw$ in $E(2)^-$}
\label{F:E2-}}

The cork twist $\xtwo$ associated to this embedding of $\bw$ can be drawn by interchanging the dot on the $1$-handle with the zero-framing on $a'$, and this evidently contains an embedded $+1$ sphere -- the core of $h_b$ capped off with a disk in $B^4$ -- and several $-1$ spheres.  As a consequence, $\xtwo$ splits off an $\sss$ summand (in fact $\xtwo\cong\bx_{3,20}$~\cite{akbulut:contractible}) and so has vanishing Seiberg-Witten invariants.  Moreover, this summand is disjoint from the nucleus $N(2)$, and it follows (as first observed by Akbulut and Yasui~\cite[Theorem 1.2]{akbulut-yasui:knotting-corks}) that $\bw$ has {\em infinitely many} distinct cork embeddings in $\xtwo$; these can be constructed using knot surgeries~\cite{fs:knots} on a regular fiber in $N(2)$, appealing to the fact that such surgeries are all $1$-stably equivalent~\cite{akbulut:fs-knot-surgery,auckly:stable}.

By a completely analogous argument it is seen that $\bw$ embeds in $E(n)^-$ for all $n\ge2$ (the only change in the picture is to add more $\pm1$-framed $2$-handles and adjust the framing to $-n$ on the meridian to the trefoil, see~\cite{bizaca-gompf:elliptic,akbulut-yasui:knotting-corks}) and $\xn$ (presumably diffeomorphic to $\bx_{2n-1,10n}$) splits off an $\sss$-summand, and so has trivial Seiberg-Witten invariants.  As before, there is a Gompf nucleus $N(n)$ disjoint from this embedding and from the $\sss$ summand, and so $\bw$ has infinitely many distinct cork embeddings in $\xn$.

\subsection*{\bf The corks $\ch$ and $\chbar$}\label{S:positron} 

Many other corks have been discovered,  detected in a similar fashion~\cite{akbulut-yasui:corks-plugs} or via contact invariants of their boundaries~\cite{akbulut-matveyev:convex}.  In \figref{corks} we introduce two infinite families of corks, denoted $\ch$ and $\chbar$ for $h$ a nonpositive half-integer (strictly negative in the latter case), that give rise to infinite families of $1$-stably isotopic $2$-spheres in completely decomposable manifolds.  The boundary involution $\tau$ is indicated in the figure.  

\fig{55}{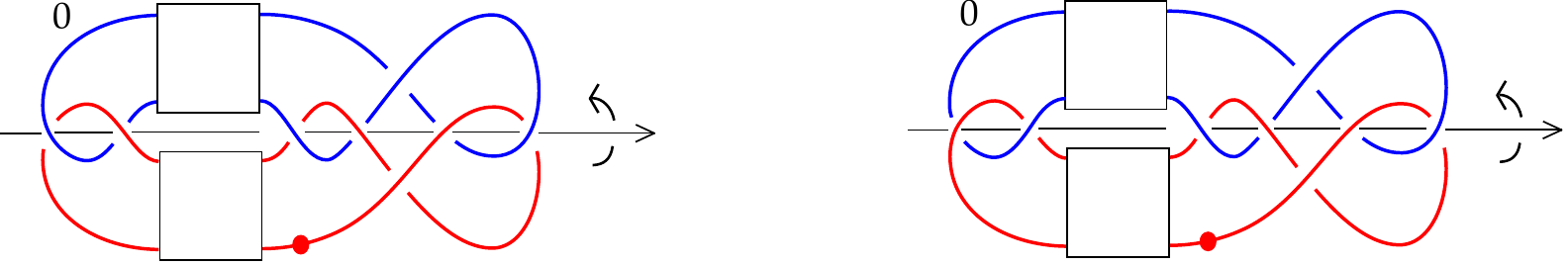}{
\put(-280,-14){\footnotesize$\ch$}
\put(-80,-14){\footnotesize$\chbar$}
\put(-200,40){\footnotesize$\tau$}
\put(-12,40){\footnotesize$\tau$}
\put(-285,40){\footnotesize$h$}
\put(-285,10){\footnotesize$h$}
\put(-97,40){\footnotesize$h$}
\put(-97,10){\footnotesize$h$}
\caption{Corks $\ch$ and $\chbar$}
\label{F:corks}}

\noindent The difference between $\ch$ and $\chbar$ is that the positive full twist on the left of $\ch$ becomes a negative twist in $\chbar$.  Evidently $C_0$ is the Akbulut-Mazur cork $\bw$, and $\cbar_{-1/2}$ is the {\em positron} $\bp$, introduced in~\cite{akbulut-matveyev:convex} and denoted there by $W_1$ and by $\ol{W}_1$ in~\cite{akbulut-yasui:corks-plugs}; see \figref{positron}.
(Note: When  referring to a cork $(C,\tau)$, we sometimes suppress the boundary involution $\tau$ from the notation if it is clear from the context.)   

\newpage

\fig{80}{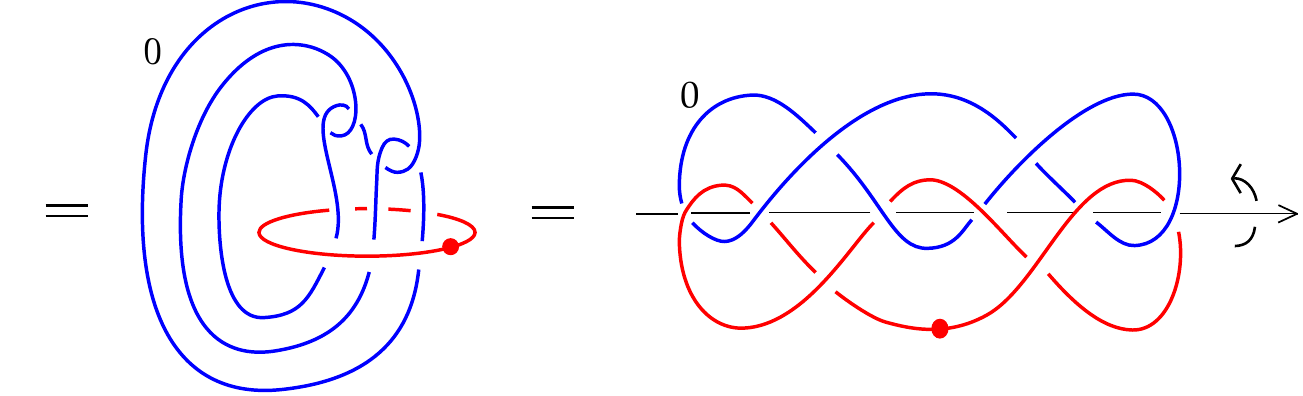}{
\put(-270,35){$\bp$}
\put(-10,45){\footnotesize$\tau$}
\caption{The positron $(\bp,\tau)$}
\label{F:positron}}

That $\ch$ and $\chbar$ are indeed corks is seen as follows (or can be shown easily by the methods of \cite{akbulut-matveyev:convex}):  The cork $\ch$ embeds in the blown-up Akbulut-Mazur cork $\bw\cs |2h|\,\cptwobar$, as indicated in \figref{corkemb}.  

\fig{76}{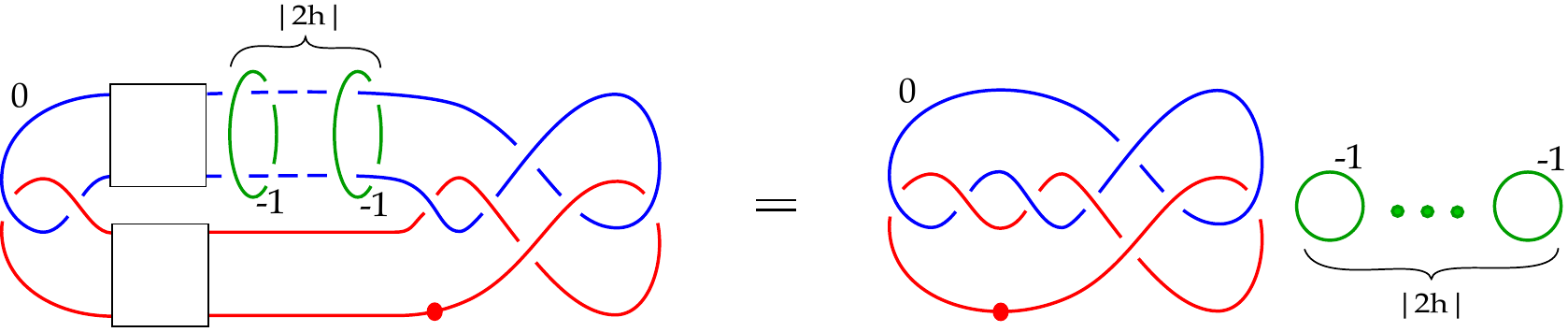}{
\put(-327,42){\footnotesize$h$}
\put(-327,10){\footnotesize$h$}
\caption{The embedding $\ch\subset \bw\cs |2h|\,\cptwobar$}\label{F:corkemb}}

\noindent Since $\bw$ is a cork in $E(n)\cs\cptwobar$ for every $n\ge2$, and further (negative) blowups do not kill the Seiberg-Witten invariants, it follows that $\ch$ is a cork in $E(n)\cs(|2h|+1)\cptwobar$.

\smallskip

Similarly \figref{corkbaremb} shows that $\chbar$ embeds in the blown-up positron $\bp\cs(|2h|-1)\,\cptwobar$.  

\fig{76}{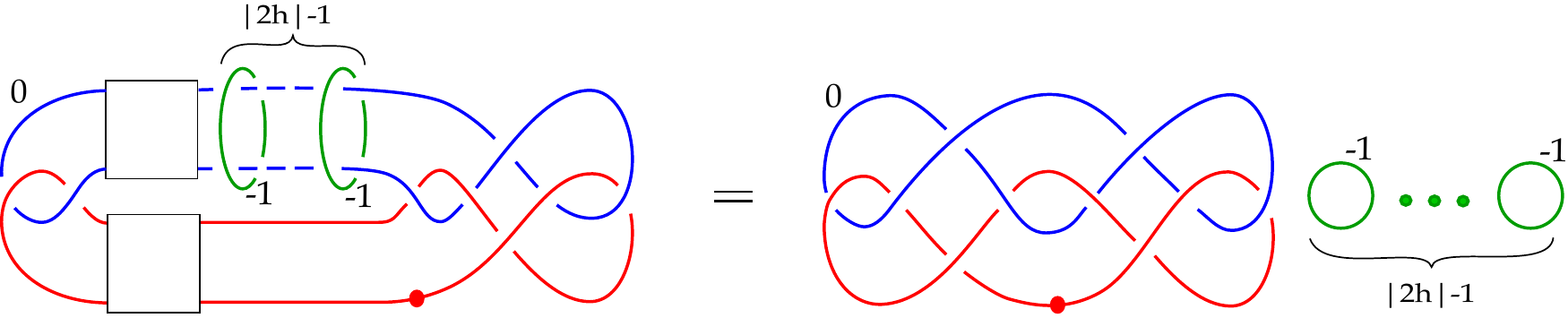}{
\put(-342,42){\footnotesize$h$}
\put(-342,10){\footnotesize$h$}
\caption{The embedding $\chbar\subset \bp\cs (|2h|-1)\,\cptwobar$}\label{F:corkbaremb}}

\noindent
It has been shown by Akbulut~\cite{akbulut:dolgachev} that $\bp$ embeds in the Dolgachev surface $E(1)_{2,3}$ with associated cork twist $E(1) \cong \bx_{1,9}$.  It follows that $\chbar$ is a cork in $E(1)_{2,3}\cs (|2h|-1)\,\cptwobar$, and also in $\bx_{1,8+|2h|}$.


\section{$1$-stably symmetric corks and the proofs of Theorems \ref{T:spheres} and \ref{T:links}}\label{S:symmetric}

We will prove our stable isotopy results by showing that the Akbulut-Mazur cork $\bw = C_0$ and the positron $\bp = \smash{\ol{C}_{-1/2}}$ become symmetric (in a precise sense defined below) after summing with $\cptwo$.  The argument extends with no extra effort to the families $\ch$ and $\smash{\chbar}$, and so we have drawn the pictures for the general case.

Given any cork $(C,\tau)$, the boundaries of $C$ and $C^+ = C\cs\cptwo$ are naturally identified, so $\tau$ can also be viewed as an involution on $\partial C^+$.  For the remainder of this section, we denote by $S$ the copy of $\cpone$ in the $\cptwo$ factor of $C^+$.  

\begin{definition*}\label{d:stablysymmetric}
A cork $(C,\tau)$ is {\em $1$-stably symmetric} 
if the involution $\tau$ on  $\partial C^+$ extends to a smooth automorphism (still called $\tau$) of $C^+$ so that the $2$-spheres $S$ and $T = \tau(S)$ are $1$-stably isotopic, i.e.\ smoothly isotopic in $C^+\!\!\cs\sss$.
\end{definition*}

There is no particular reason to suppose that the extension of $\tau$ is an involution, although it turns out to be so in the examples presented below. Our basic technique for constructing strictly $1$-stably isotopic spheres is  summarized in the following theorem.

\begin{theorem}\label{T:1-stable}
If $(C,\tau)$ is a $1$-stably symmetric cork in a simply-connected $4$-manifold $X$, then the spheres $S$ and $T$ in $C^+$, regarded as submanifolds of $X^+$, are
\\[3pt]
\indent {\bf a)} topologically isotopic, and 
\\[3pt]
\indent {\bf b)} strictly $1$-stably isotopic, i.e. smoothly isotopic in $X^+\#\sss$ but not in $X^+$.
\vskip5pt
\noindent Furthermore, if $(C,\tau)$ has a family of distinct cork embeddings in $X$ with associated cork twists $X_k$, for $k$ in some indexing set, then there is a corresponding family of topologically isotopic spheres $S_k$ in $X^+$ that are pairwise strictly $1$-stably isotopic.
\end{theorem}

\begin{proof}
We sketch the standard argument for {\bf a)}; compare~\cite{lee-wilczynski:flat,kim-ruberman:surfaces,sunukjian:surfaces}.  Since $S$ and $T$ are embedded in $C^+$ with simply-connected complements, they are homologous, from which it follows (using Freedman~\cite{freedman-quinn}) that there is a self-homeomorphism of $X^+$ taking $S$ to $T$.  This homeomorphism may be assumed to be the identity on $H_2(X^+)$, so the theorem of Perron~\cite{perron:isotopy2} and Quinn~\cite{quinn:isotopy} implies that it is topologically isotopic to the identity, yielding a topological ambient isotopy taking $S$ to $T$.

To prove {\bf b)}, note that since $(C,\tau)$ is $1$-stably symmetric, the spheres $S$ and $T$ are (smoothly) isotopic in $C^+\!\!\cs\sss$, and hence also in  $X^+\!\!\cs\sss$.  But they are not isotopic in $X^+$, since blowing down $S$ in $X^+$ yields $X$ while blowing down $T$ yields $X^\tau$:
$$
\begin{aligned}
X^+/\,T \ &= \ (X - C^\circ) \cup_{\textup{id}} (C^+/\,T) \\ 
&\cong \ (X - C^\circ) \cup_{\tau} (C^+/\,S) \ = \ (X - C^\circ) \cup_{\tau} C \ = \  X^\tau\,.
\end{aligned}
$$
Since $C$ is a cork in $X$, the manifolds $X^+/\,S$ and $X^+/\,T$ are not diffeomorphic, so there is no diffeomorphism of  $X^+$ taking $S$ to $T$.  (In principle, this is a stronger statement than saying that $S$ and $T$ are not isotopic, but we don't know an example to illustrate the difference.)

We now prove the last assertion.  By hypothesis, there are embeddings $C_k\subset X$ of copies of $C$ with $X^{\tau_k}\cong X_k$, where $\tau_k:\partial C_k\toself$ is the involution corresponding to $\tau$.  After suitable isotopies, we may assume that the intersection of the $C_k$'s has nonempty interior, and a positive blowup of a point there yields a sphere $S\subset X^+$ of self-intersection $1$ that lies in each $C_k$.  Since $C$ is $1$-stably symmetric, $\tau_k$ extends over $C_k^+\subset X^+$, and we set 
$$
S_k \ = \ \tau_k(S) \ \subset \ X^+.
$$  
Now the argument that the $S_k$ are topologically isotopic and pairwise strictly $1$-stably isotopic proceeds exactly as in {\bf a)} and {\bf b)}, noting that $X^+/S_k \cong X^{\tau_k} \cong X_k$. \end{proof}

\begin{theorem}\label{T:stablysymmetric}
The corks $\ch$ and $\chbar$ are $1$-stably symmetric.
\end{theorem}

\begin{proof}
The arguments for $\ch$ and for $\chbar$ are effectively the same, so we will treat the former in detail, and only briefly discuss the latter at the end of the proof. 

We first show that the involution $\tau$ on $\partial\ch$ extends after a positive blowup. This fact is well known for the Akbulut-Mazur cork $C_0$ (and implicit in \cite{akbulut-kirby:mazur}) but we will give a careful argument so that we may draw the spheres, labeled $\sh$ and $\tsh = \tau(\sh)$, in a suitable picture of $\ch^+ = \ch\cs\cptwo$.  This will allow us to explicitly construct the required isotopy between $\sh$ and $\tsh$ in $\ch^+\cs\sss$.  

Our initial picture of $\ch$ (Figure~\ref{F:corks}) has the evident boundary involution $\tau$ given by $180^\circ$ rotation about the midlevel horizontal axis.  Our final picture of $\ch^+$ will have the analogous boundary involution, that we continue to label $\tau$, whose extension will be apparent from the picture.  As the bounding $3$-manifolds $\partial\ch^+$ carry more than one involution~\cite{bonahon-siebenmann,henry-weeks,kodama-sakuma}
we will be careful to verify that the sequence of handle slides that carry one picture to the other are $\tau$-equivariant on the boundary, so that these involutions may be identified.   

Let $\kh$ be the knot in $S^3$ drawn in \figref{ribbon} below, and $\rh$ be the $4$-manifold obtained from $B^4$ by attaching a $+1$ framed $2$-handle along $\kh$.  The indicated rotational symmetry $\tau$ of $\kh$ is induced by a linear involution of $B^4$ that extends in a standard fashion to $\rh$.

\fig{75}{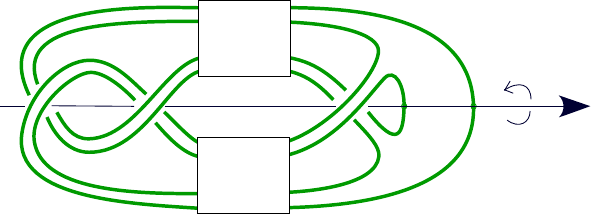}{
\put(-20,45){\small$\tau$}
\put(-133,60){\footnotesize$h-\frac12$}
\put(-133,12){\footnotesize$h-\frac12$}
\caption{The ribbon knot $K_h$}
\label{F:ribbon}}

\noindent Observe that $\kh$ is a ribbon knot.  In particular, it bounds two obvious immersed ribbon disks $\ah$ and $\bh$ in $S^3$, with $\tau(\ah) = \bh$.  These disks are drawn on the left and right side of \figref{sh} for the case when $h$ is an integer, each appearing as a pair of stacked disks, twisted in the middle and joined by a single ribbon.   Resolving the singularities of $\ah$ and $\bh$ by pushing their interiors into $B^4$, and then capping off with the core of the $2$-handle, we obtain a pair of homologous embedded $2$-spheres $\sh$ and $\tsh$ in $\rh$ with $\tsh = \tau(\sh)$, specified by the same pictures by our drawing conventions. 

\fig{70}{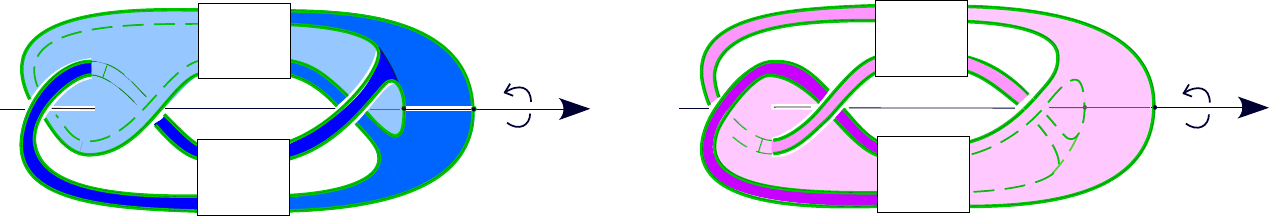}{
\put(-412,12){\footnotesize$A_h$}
\put(-193,12){\footnotesize$B_h$}
\put(-266,2){\large$S_h$}
\put(-45,4){\large$T_h$}
\put(-270,62){\small$1$}
\put(-50,62){\small$1$}
\put(-341,55){\footnotesize$h-\frac12$}
\put(-341,12){\footnotesize$h-\frac12$}
\put(-124,55){\footnotesize$h-\frac12$}
\put(-124,12){\footnotesize$h-\frac12$}
\caption{The ribbon disks $A_h, B_h$ and associated spheres $\sh, \tsh$ in $\rh$}
\label{F:sh}}

\vskip-20pt
\vskip-20pt
   
\begin{lemma}\label{L:ribbon}
There is a diffeomorphism $\rho:\ch^+ \to \rh$ that is $\tau$-equivariant on the boundary and that maps $S\subset\ch^+$ to the sphere $\sh\subset\rh$ shown on the left in \textup{\figref{sh}}.
\end{lemma}

Before proving the lemma, observe that it yields the desired extension of $\tau$ over $\ch^+$.  Simply conjugate the extension of $\tau$ over $\rh$ by $\rho$.  Furthermore, the last assertion in the lemma shows that the triple $(\rh,\sh,\tsh)$ is diffeomorphic to $(\ch^+,S,T)$, so the proof of the theorem will be complete once we produce an isotopy between $\sh$ and $\tsh$ in $\rh\cs\sss$.  To accomplish this, we apply the `key stable isotopy' shown in \figref{key} and simultaneously keep track of the motion of $\sh$ and $\tsh$.  The stabilization adds a pair of $2$-handles attached along a zero-framed Hopf link.  We then slide the $+1$ framed handle attached to the knot $\kh$ four times over one of these $2$-handles, moving $\kh$ to an unknot in $S^3$, and $\sh$ and $\tsh$ to a pair of disks in $S^3$ bounded by this unknot, as shown in \figref{isotopy} after a further isotopy to remove canceling intersections of the ribbons with the stacked disks.  Since any two disks bounding an unknot in $S^3$ are isotopic rel boundary, this yields the desired isotopy.  Note that the initial position of each of the surfaces $S_h$ after sliding $K_h$ over the Hopf link is far more complicated; the  point of the isotopy in \figref{key}\, is to achieve the simple position shown in \figref{isotopy}.

\fig{70}{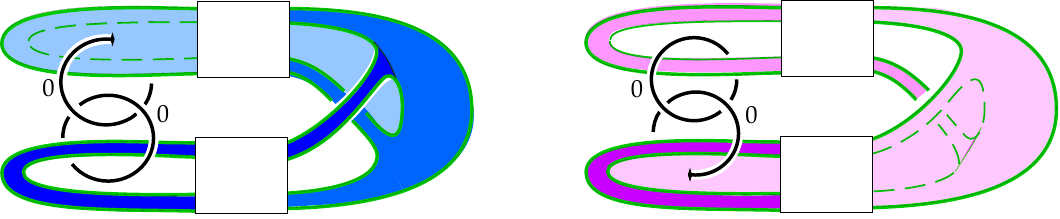}{
\put(-350,6){$\sh$}
\put(-10,6){$\tsh$}
\put(-85,55){\footnotesize$h-\frac12$}
\put(-85,12){\footnotesize$h-\frac12$}
\put(-272,55){\footnotesize$h-\frac12$}
\put(-272,12){\footnotesize$h-\frac12$}
\caption{$\sh$ and $\tsh$ in $\rh\cs\sss$ after the key stable isotopy}
\label{F:isotopy}}

It remains to prove the lemma.  We will first define the diffeomorphism $\rho:\ch^+\to\rh$ as the composition of three explicit intermediate diffeomorphisms $\rho_1,\rho_2,\rho_3$\,, each $\tau$-equivariant on the boundary, as indicated in the commutative diagram in \figref{rho}.  (For the moment, ignore the $1$-handles given by the small dotted circles.)  We will then verify that $\rho(S) = \sh$.  

\fig{200}{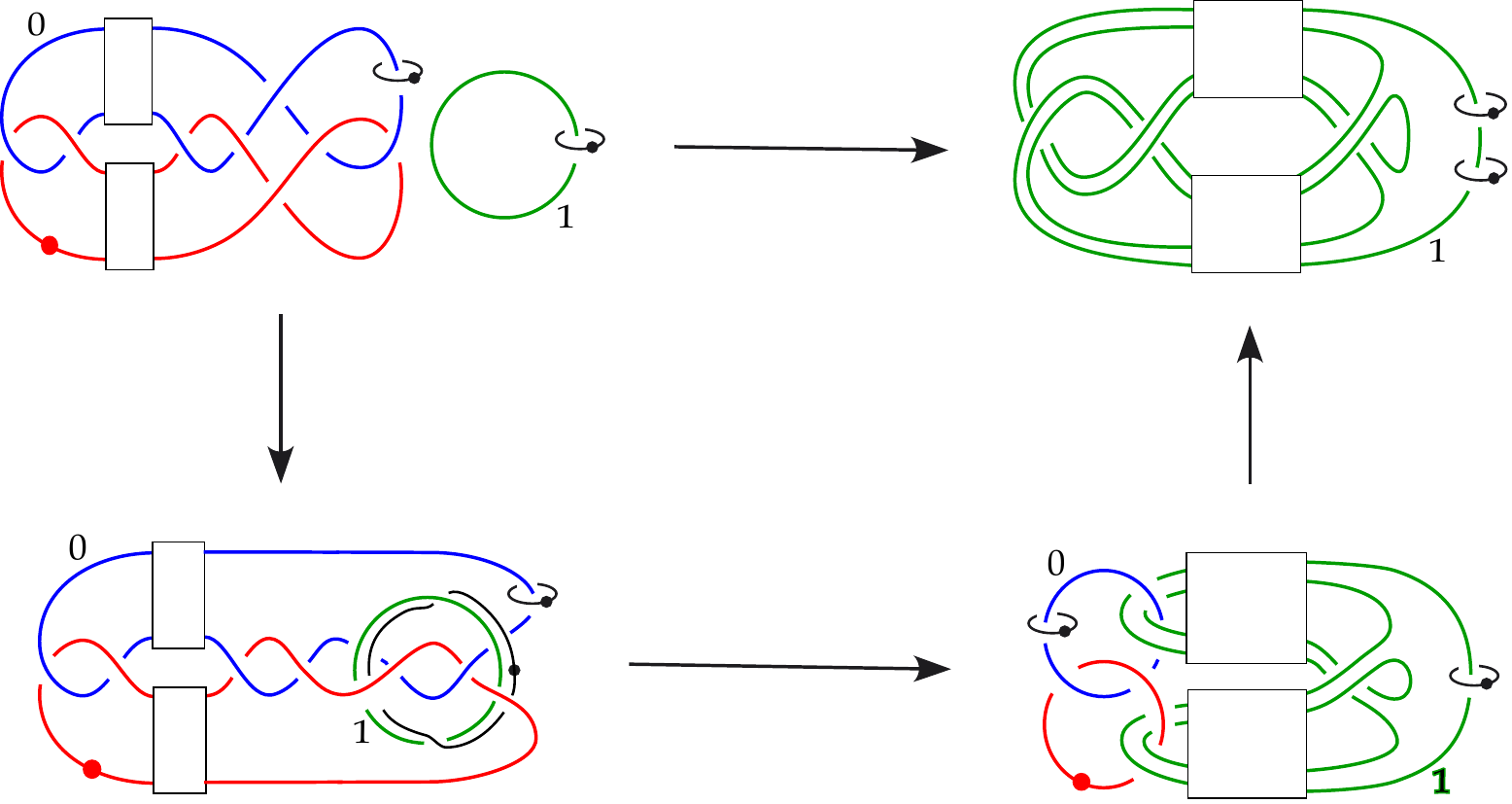}{
\put(-77,45){\footnotesize$h-\frac12$}
\put(-77,10){\footnotesize$h-\frac12$}
\put(-77,185){\footnotesize$h-\frac12$}
\put(-77,142){\footnotesize$h-\frac12$}
\put(-334,50){\footnotesize$h$}
\put(-334,10){\footnotesize$h$}
\put(-347,182){\footnotesize$h$}
\put(-347,142){\footnotesize$h$}
\put(-236,51){\footnotesize$\mu$}
\put(-245,32){\footnotesize$\nu$}
\put(-270,186){\footnotesize$\mu$}
\put(-224,167){\footnotesize$\nu$}
\put(-324,100){$\rho_1$}
\put(-295,105){blowdown}
\put(-297,95){and blowup}
\put(-55,95){$\rho_3$}
\put(-112,100){cancel}
\put(-115,90){$1/2$-pair}
\put(-190,20){$\rho_2$}
\put(-185,170){$\rho$}
\put(-205,55){equivariant}
\put(-195,45){isotopy}
\caption{The diffeomorphism $\rho = \rho_3\circ\rho_2\circ\rho_1:\ch^+\to\rh$}
\label{F:rho}}

\vskip-10pt
Here are the details.  To define $\rho_1$, draw the handlebody picture for the blown up cork $\ch^+$ by adding an unlinked $+1$-framed unknot to the picture of $\ch$ in \figref{corks}, as shown in the upper left corner of \figref{rho}.  Alternatively, one could add a linking $+1$-framed circle to cancel the left half twists on the top and bottom, as shown in the lower left corner.  Now there is a $\tau$-equivariant diffeomorphism $\rho_1$ between the boundaries of these two handlebodies, given by blowing down the unlinked $+1$ and then blowing up the linked $+1$.  This map carries the $0\textup{-framed}$ meridians $\mu,\, \nu$ of the $2$-handles on top to the correspondingly labeled curves on the bottom.  Putting dots on these curves, i.e.\ treating them as $1$-handles, corresponds to removing tubular neighborhoods of the trivial disks that they bound in $B^4$.  These $1$-handles clearly cancel the $2$-handles, leaving a single $1$-handle in both cases, so by the Extension \lemref{extend}\,, $\rho_1$ extends to a diffeomorphism between the two handlebodies.

Next define $\rho_2$ as the end of a $\tau$-equivariant isotopy of the attaching maps, as detailed in \figref{equivariant}.  The first step $\smash{\rho_2^\circ}$, which may be hard to visualize, is broken down in \figref{detail}.       

\fig{55}{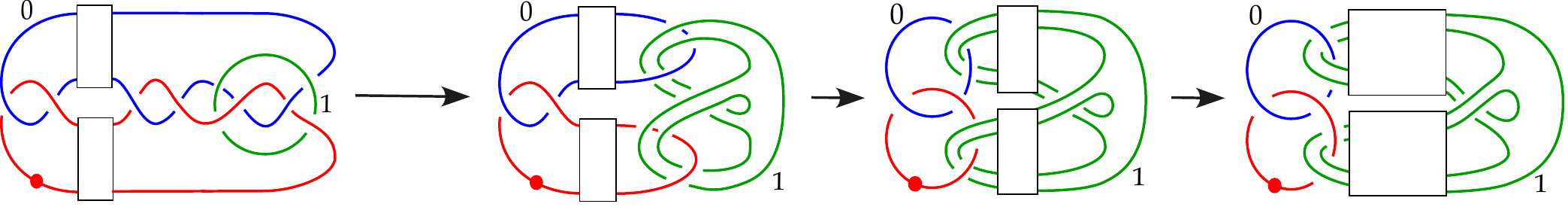}{
\put(-149,40){\footnotesize$h$}
\put(-149,10){\footnotesize$h$}
\put(-262,41){\footnotesize$h$}
\put(-261,10){\footnotesize$h$}
\put(-396,42){\footnotesize$h$}
\put(-396,10){\footnotesize$h$}
\put(-56.5,40){\footnotesize$h-\frac12$}
\put(-56.5,11){\footnotesize$h-\frac12$}
\put(-315,15){$\rho_2^\circ$}
\caption{$\rho_2$: isotope the attaching maps}
\label{F:equivariant}}

\fig{75}{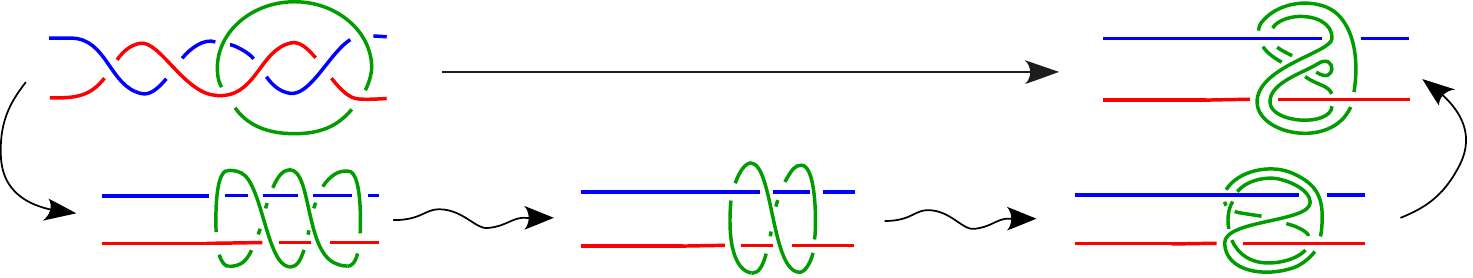}{
\put(-200,62){$\rho_2^\circ$}
\caption{Breakdown of $\rho_2^\circ$}
\label{F:detail}}

Finally define $\rho_3$ by canceling the $1$-handle and the $0$-framed $2$-handle, i.e.\ slide the upper `band' of the $+1$ circle over the $1$-handle and the lower band over the $0$-framed $2$-handle in an equivariant fashion, and then cancel the $1/2$-handle pair to get $\rh$; see the righthand side of \figref{rho}\,, ignoring the dotted meridians as before.  Tracking the meridians as with $\rho_1$ confirms, using the Extension \lemref{extend}\,, that this process defines a diffeomorphism $\rho_3$ that is $\tau$-equivariant on the boundary.

It remains to verify that $\rho(S) = \sh$, where $S$ is the exceptional sphere in the blowup $\ch^+$.  Clearly $S$ can be viewed as the union of the core of the $+1$-framed $2$-handle in the top left drawing in \figref{rho} with a trivial spanning disk for its attaching circle), as shown in the top left corner of \figref{rhosphere}.  Now we simply track $S$ through the rest of the diagram in \figref{rho}\,, as shown in \figref{rhosphere}.  In particular, \figref{detailsphere} provides the details of the transformation of $S$ through the sequence of moves in \figref{detail}\,.

\fig{200}{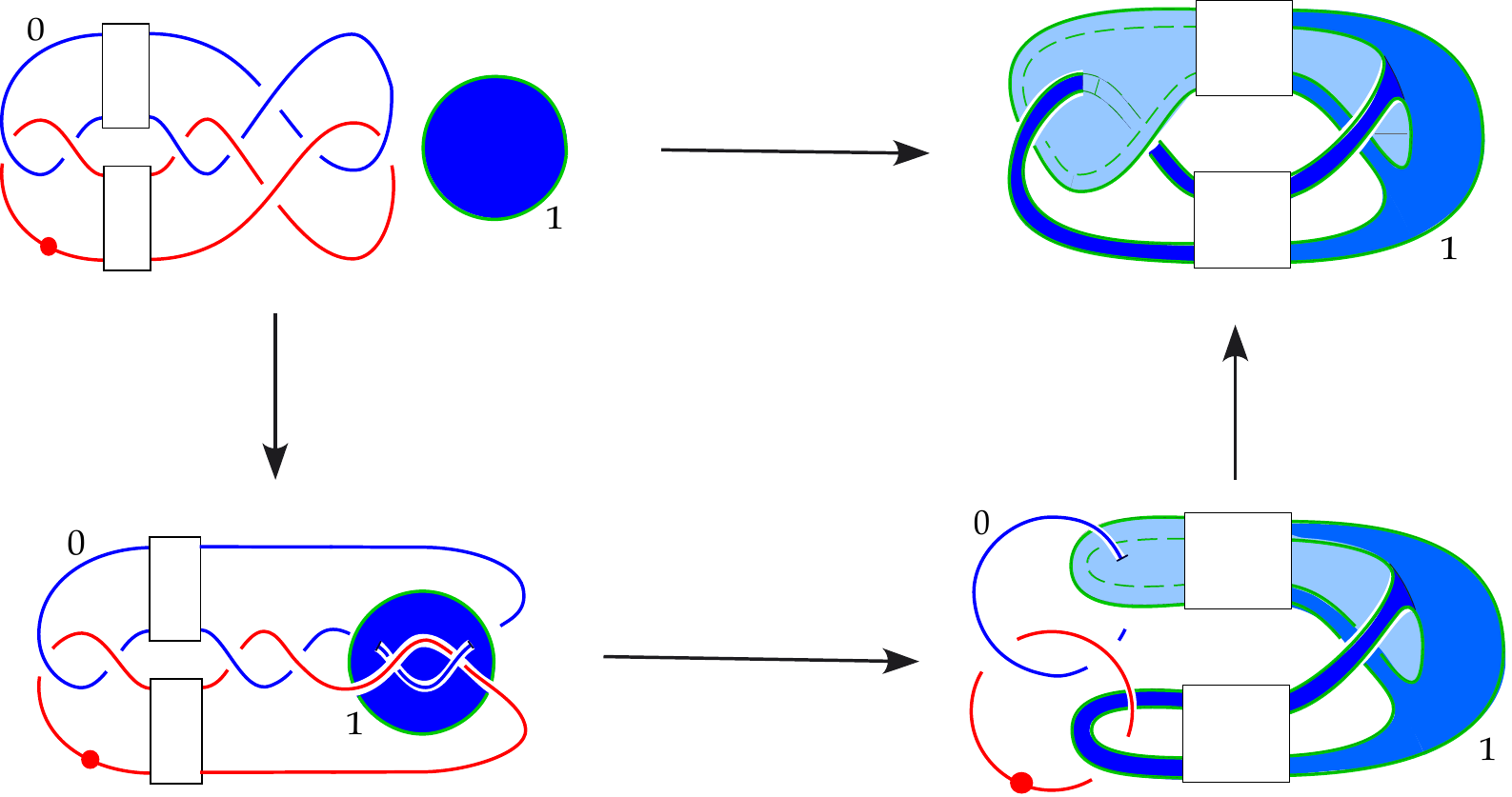}{
\put(-79,55){\footnotesize$h-\frac12$}
\put(-79,10){\footnotesize$h-\frac12$}
\put(-78,185){\footnotesize$h-\frac12$}
\put(-78,142){\footnotesize$h-\frac12$}
\put(-338,50){\footnotesize$h$}
\put(-338,13){\footnotesize$h$}
\put(-351,181){\footnotesize$h$}
\put(-351,142){\footnotesize$h$}
\put(-326,100){$\rho_1$}
\put(-61,95){$\rho_3$}
\put(-195,20){$\rho_2$}
\put(-185,170){$\rho$}
\put(-240,177){$S$}
\put(-5,180){$S_h$}
\caption{Tracking the sphere to show $\rho(S) = S_h$}
\label{F:rhosphere}}

\fig{75}{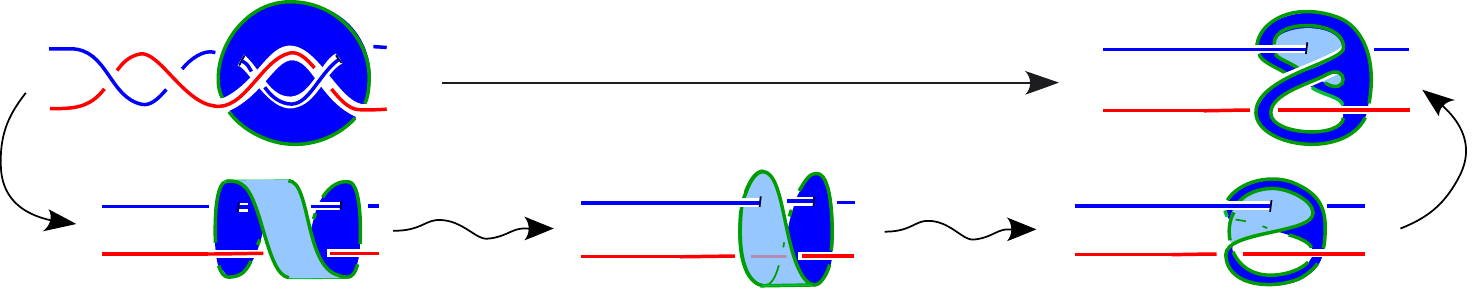}{
\put(-200,62){$\rho_2^\circ$}
\caption{Tracking the sphere through $\rho_2^\circ$}
\label{F:detailsphere}}

This proves \lemref{ribbon}\,, and hence \thmref{stablysymmetric}\,, for $\ch$.  The argument for $\chbar$  is  completely analogous, except that the knot $\kh$ is replaced by the knot $\khbar$ shown in \figref{ribbonbar}\,.

\fig{75}{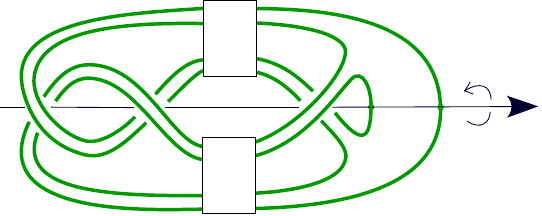}{
\put(-20,47){\small$\tau$}
\put(-112,60){\footnotesize$h$}
\put(-112,12){\footnotesize$h$}
\caption{The ribbon knot $\khbar$}
\label{F:ribbonbar}}
\vskip-.3in
\end{proof}



We can now prove our main stabilization results about spheres.

\subsection*{\bf Proof of \thmref{spheres}}\label{S:A}

By  \thmref{stablysymmetric}\,, the Akbulut-Mazur cork $\bw = C_0$ and the positron $\bp = \ol C_{-1/2}$ are $1$-stably symmetric corks.  Thus by \thmref{1-stable}, it suffices to show that for the stated values of $p$ and $q$,  either $\bw$ has infinitely many distinct cork embeddings (when $p\ge4$) or $\bp$ has an effective cork embedding (when $p=2$) in some $4$-manifold $X$ with $X^+ \cong \xpq$.  But by \cite{akbulut-yasui:knotting-corks}, $\bw$ has infinitely many distinct embeddings in $\xn$ (see \secref{cork}), and $E(n)^{-{\tau}+}\cong E(n)^{-+}$ (since $\tau$ extends to $\bw^+$) which is diffeomorphic to $\bx_{2n,10n}$ by the Mandelbaum-Moishezon trick, as noted at the end of \secref{prelim}.  Since manifolds distinguished by their Seiberg-Witten invariants remain distinct after summing with $\cptwobar$~\cite{fs:sw-blowup}, this proves the theorem when $p\ge4$.  Similarly by \cite{akbulut:dolgachev}, $\bp$ is a cork in the Dolgachev surface $E(1)_{2,3}$ (see the end of \secref{cork}) and $E(1)_{2,3}^+ \cong \bx_{2,9}$ \cite{moishezon:sums}.  The proof for  $p=2$ now follows from the blowup formula for Seiberg-Witten invariants.
{\LARGE \hfill\ensuremath{\square}}

\subsection*{\bf Disks versus spheres}\label{S:pi1}

Viewing the spheres $\sh$ and $\tsh$ in the model for the blown up cork $\ch^+$ shown in the bottom right corner of \figref{rhosphere} provides good intuition into the result that we just proved.   In this model $\sh$ is plainly visible, as shown in the figure, while $\tsh$ will appear as the rotated image of $\sh$ with a pair of tubes added (capped off over the zero-framed $2$-handle) to avoid the disk bounding the dotted circle that was removed.  This disk evidently obstructs the isotopy of $\tsh$ to the $\sh$. When one stabilizes, the effect is to change the dotted circle to a zero-framed circle. This fills in the missing disk and the key isotopy of \secref{key} may then be used.

Even though the meridian of the disk associated to the dotted circle does not bound a smooth embedded disk, it does bound an smoothly immersed disk $D$ since the complement of $\tsh$ is simply-connected.  Presumably $D$ can be used to construct a Casson handle, which in turn can be used to construct a topological isotopy between the pair of spheres. 

One may wonder if this behavior already takes place at the level of the two ribbon disks $A_h$ and $B_h$ for the knot $K_h$ shown in \figref{sh}, i.e.\ these disks are distinct, but perhaps they become isotopic after stabilizing with $\sss$. The answer is no.  This is shown in the following proposition, and demonstrates the need for care when performing stabilization by replacing a dotted circle by a zero framed circle.

\begin{proposition}
The two ribbon disks $A_h,\,B_h \subset B^4$ for $K_h$ remain non-isotopic relative to the boundary after any number of stabilizations of $B^4$.  In fact, for any simply-connected $4$-manifold $X$, no (smooth) automorphism of $B^4\cs X$ extending the identity on the boundary will carry $A_h$ to $B_h$. 
\end{proposition}

\begin{proof}
Assume to the contrary that such an automorphism $\phi$ of $B^4\cs X$ exists.  We will show by fundamental group considerations that this leads to a contradiction.

Recall that the involution $\tau$ of $B^4$ (introduced in the discussion above \figref{ribbon}) swaps the disks $A_h$ and $B_h$.  After an isotopy, we may assume that $\tau$ leaves fixed a small $4$-ball away from $K_h$, and use this to construct an automorphism $\tau\cs\idx$ of $B^4\cs X$ that swaps $A_h$ and $B_h$ (see \secref{stabilizing}).  Now consider the composition 
$$
\rho \ = \ (\tau\cs\idx)\circ\varphi:B^4\cs X\to B^4\cs X.
$$
This diffeomorphism preserves $A_h$ (setwise) and restricts to $\tau$ on the boundary $S^3$, inducing an automorphism $\rho_*$ of $\pi_1(B^4\cs X-A_h)=\pi_1(B^4-A_h)$ (by the Seifert-Van\,Kampen Theorem) that restricts to $\tau_*$ on $\pi_1(S^3-K_h)$.  But this is impossible, as we now show.  

The first step is to construct a handle decomposition of the complement of an open tubular neighborhood $N(A_h)$ of $A_h$ in $B^4$, following the standard technique~\cite{melvin:thesis,akbulut-kirby:branch,gompf-stipsicz:book}.  Start with the picture of $K_h$ from \figref{ribbon}\,, reproduced on the left side of \figref{ribboncomplement}\,.  Note that the curves $a$ and $b$ shown there are swapped by $\tau$.  Now $B^4-N(A_h)$ is obtained from a small $4$-ball about the origin by adding an $(i+1)$-handle for each $i$-handle in $A_h$ associated to any critical level embedding, and so in our case we have two $1$-handles and one $2$-handle, as shown in the middle picture in \figref{ribboncomplement}.  This picture is redrawn after an isotopy on the right side of the figure, also tracking the $a$ and $b$ curves through this process.

\fig{70}{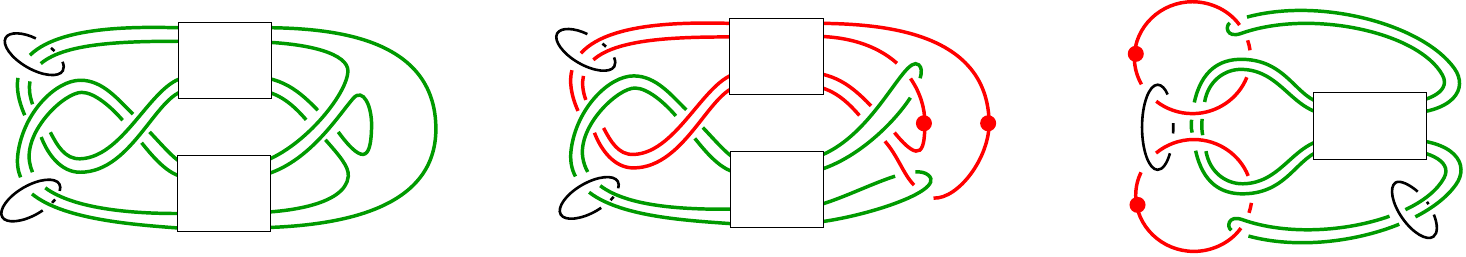}{
\put(-405,60){\small$a$}
\put(-405,5){\small$b$}
\put(-255,60){\small$a$}
\put(-253,5){\small$b$}
\put(-95,33){\small$a$}
\put(-4,1){\small$b$}
\put(-349,52){\footnotesize$h-\frac12$}
\put(-349,16){\footnotesize$h-\frac12$}
\put(-199,52){\footnotesize$h-\frac12$}
\put(-199,16){\footnotesize$h-\frac12$}
\put(-40,33){\footnotesize$2h-\frac12$}
\caption{Complement of the ribbon disk $A_h$}
\label{F:ribboncomplement}}

One can now read off a presentation of the fundamental group of the complement of $A_h$
$$
\pi_1(B^4-A_h) \ = \ \langle\, x,y \,\st\, xy = (yx)^2\,\rangle
$$
and observe that $a=xy$ and $b=1$.  By hypothesis $\rho_*(a)$ is conjugate to $b$, so $xy=1$.  But this group has a representation onto the symmetric group $S_3$ mapping $x$ to $(1\,2)$ and $y$ to $(2\,3)$, and  so $xy$ to $(1\,2\,3)$, a contradiction.   
\end{proof}

\break

\begin{remark*}\label{R:am-positron}
This fundamental group calculation is specific to the family of corks $\ch$.  In contrast, the complements of the ribbon disks associated to the positron cork $\bp$ and its generalizations $\chbar$ (given by the same picture but with $2h$ in the twist box) have fundamental group $\langle\, x,y \st xy = x^2y^2\,\rangle\,\cong\,\bz$.  These disks are topologically equivalent but smoothly distinct, and give rise to the phenomenon that Akbulut calls~\cite[\S 10.2]{akbulut:book} (see also \cite{akbulut:zeeman}\cite{akbulut:spheres}) an `anti-cork'.  We do not know if they are $1$-stably isotopic. 
\end{remark*}

\subsection*{\bf Brunnian links and the proof of \thmref{links}}\label{S:brunnian}

A non-trivial link in the sphere all of whose proper sublinks are trivial is called Brunnian~\cite{brunn,debrunner:brunnian}.  The notion has an obvious extension to links in arbitrary manifolds; one might look for families of inequivalent links with equivalent proper sublinks. As in \secref{intro}\,, we call such a family {\em Brunnian}.  Here we prove \thmref{links}\,, asserting the existence of Brunnian families of $m$-component links for any $m\ge2$ in suitable completely decomposable $4$-manifolds, using \thmref{stablysymmetric}.

We first construct Brunnian pairs $(L,M)$.  Start with a $1$-stably symmetric cork $C$ in 4-manifold $X$, and assume that $C$ is also a cork in $X\cs2\,\cptwobar$.  By the blowup formula, this is automatically the case whenever $X$ and $X^\tau$ are distinguished by Seiberg-Witten or Donaldson invariants.  The spheres $S$ and $T$ in $C^+ \subset X^+$ 
yield links $L = (S, S_2,\ldots, S_m)$ and $M = (T, S_2,\ldots, S_m)$ in $Z = X^+ \cs (m-1)\,\cptwo \cs2\,\cptwobar$, where $S_2,\ldots, S_m$ are the exceptional spheres in the $\cptwo$'s.  Evidently $L$ and $M$ are smoothly distinct, since $Z/L = X^{--}$\,, $Z/M = X^{\tau--}$, and $Z/L\not\cong Z/M$ because $C$ is a cork in $X^{--}$.  Since twisting is not seen topologically, $L$ and $M$ are topologically isotopic.

To prove the Brunnian property, it suffices by symmetry to show that the proper sublinks $L' = (S,S_2,\dots,S_{m-1})$ and $M' = (T,S_2,\dots,S_{m-1})$ are isotopic in $Z$.  Grouping the last three summands in $Z$ with $X^+$, we view $Z = X^{++--} \# (m-2)\,\cptwo$ with $S$ and $T$ lying in $X^{++--} = X^+\cs\sss\cs\cptwobar$ and $S_2,\dots,S_k$ lying in $(m-2)\cptwo$.  The $1$-stable symmetry of $C$ implies that $S$ and $T$ are isotopic in $X^{++--}$, whence $L'$ and $M'$ are isotopic in $Z$.

To get the statement of \thmref{links}, suppose first that $p=2n$ with $n\geq 2$.  If we choose $X$ to be the blown-up elliptic surface $E(n)^-$, and use the Akbulut-Mazur cork $\bw$ in the construction described above, we have $Z \cong\bx_{p+m-1,5p+2}$.  Values of $q > 5p+2$ result from blowing up $E(n)^-$ an arbitrary number of times.  If $n=1$, then we would choose $X$ to be the Dolgachev surface, and let $C$ be the positron cork $\bp$  in $X$. The construction then produces a Brunnian pair in $\bx_{m+1,11}$,  and by blowing up $X$, in $\bx_{m+1,q}$ for any $q \geq 11$.  

As in \thmref{spheres}, the fact that $\bw$  has infinitely many distinct cork embeddings in a \break 
$4$-manifold $\mathbb{X}$ with $\mathbb{X}^+\cong \mathbb{X}_{2n,10n}$ leads to infinite Brunnian families.  {\LARGE \hfill\ensuremath{\square}}  


\section{$1$-stable isotopy of diffeomorphisms and psc metrics}\label{S:diffeo}


As mentioned in the introduction, a sphere $\Sigma$ of self-intersection $\pm1$ in a $4$-manifold $Z$ gives rise to a self-diffeomorphism $\rho_\Sigma$ of $Z$ that induces `reflection in $[\Sigma]$\,' on $H_2(Z)$, and isotopic spheres yield isotopic diffeomorphisms. (This diffeomorphism is just the connected sum of the identity on the complement of the sphere with complex conjugation acting on a neighborhood of the sphere in $\pm\cptwo$.) The sequence of papers~\cite{ruberman:isotopy,ruberman:polyisotopy,ruberman:swpos} uses that observation to construct interesting self-diffeomorphisms as follows:  Start with a sphere $S$ of self-intersection $+1$ in a simply-connected $4$-manifold $X$, and consider the self-diffeomorphism 
$$
f_{\raisebox{-2pt}{$\scriptstyle S$}} = \rho_{\raisebox{-2pt}{$\scriptstyle S+E_1 + E_2$}}  \circ \rho_{\raisebox{-2pt}{$\scriptstyle S-E_1 + E_2$}}
$$
of $Z = X\cs2\,\cptwobar$, where $E_1$ and $E_2$ are the exceptional spheres in the $\cptwobar$ factors.  Here the $\pm$ signs indicate connected sums of the spheres, preserving or reversing orientation as appropriate.  Since these spheres have simply-connected complements, the sums are well-defined by \aref{welldefined}.   Now suppose that $T$ is another sphere of self-intersection $+1$ in $X$, and that $b_2^+(X)$ is even and at least $4$.  If $S$ and $T$ are homologous, then the diffeomorphisms $f_S$ and $f_T $ are topologically isotopic \cite{perron:isotopy2,quinn:isotopy}, but it is shown in \cite{ruberman:isotopy} that they are not {\em smoothly} isotopic if $X/S$ and $X/T$ are distinguished by some Seiberg-Witten or Donaldson invariant.  

As we have seen above, we can choose the manifolds $X$ and $Z$ amongst the $\xpq$, and (with suitable restrictions on $p$ and $q$) get infinitely many spheres, resulting in infinitely many isotopy classes of diffeomorphisms.  \thmref{diffs} asserts that the resulting diffeomorphisms are all isotopic after a single stabilization.  We actually prove a somewhat sharper statement, which concerns the behavior of the group $\pi_0(\diff(M))$ of isotopy classes of (orientation preserving) diffeomorphisms of $M$ under stabilization. The issue is that connected sum is not, {\em a priori}, a well-defined operation on isotopy classes of self-diffeomorphisms.  In a technical section that follows this one, we will show that that for a simply-connected $M^4$, connected sum with the identity on $\pm \cptwo$ or $\sss$ is in fact well-defined.  As a corollary, we find a well-defined homomorphism $\Phi: \pi_0(\diff(M)) \to  \pi_0(\diff(M\cs\sss)$, and prove via \thmref{diffs} that its kernel is infinite.  

\begin{proof}[Proof of {\rm \thmref{diffs}}]
For even $p\ge4$ and $q\ge 5p+2$, the proof of \thmref{spheres} yields an infinite family of topologically isotopic spheres $S_i$ of self-intersection $+1$ in $\bx_{p,q-2}$ that are not smoothly isotopic, distinguished by the Seiberg-Witten invariants of their blow-downs $\bx_{p,q-2}/S_i$, but that become smoothly isotopic in $\bx_{p,q-2}\cs\sss \cong \bx_{p+1,q-1}$.  The results of \cite{ruberman:isotopy} (discussed above) show that the self-diffeomorphisms $f_{S_i}:\xpq\to\xpq$ are pairwise non-isotopic.  Since the $S_i$ are isotopic in $\bx_{p+1,q-1}$, the spheres $S_i\pm E_1+E_2$, and consequently the diffeomorphisms $f_{S_i}$, are isotopic in $\bx_{p+1,q+1} \cong \xpq\cs\sss$, as asserted.
\end{proof}

We turn now to the proof of \thmref{psc}\,.  The $\xpq$ admit metrics of positive scalar curvature (PSC) constructed as connected sums~\cite{gromov-lawson:psc,schoen-yau:psc} of standard metrics on $\pm\cptwo$. For a generic PSC metric $g_{p,q}$ on $\xpq$, it was shown in~\cite{ruberman:swpos} that the metrics $f^*_{S_i} g_{p,q}$ are mutually non-isotopic. We want to show that these become isotopic after a single stabilization.  

The notion of stabilization of metrics runs into the same issue of well-definedness (even up to isotopy) as we encountered in the case of diffeomorphisms.  In order to perform a connected sum of PSC metrics on manifolds $X$ and $Y$ one typically deforms the metrics (maintaining the PSC condition) in a neighborhood of points $x$ in $X$ and $y$ in $Y$ into a standard form (a `torpedo metric' \cite{walsh:cobordism}) after which they can be glued by identifying $X$ and $Y$ away from $x$ and $y$.  The issue of whether the resulting metric on $X\cs Y$ is well-defined (independent up to isotopy of all choices in the construction) can be addressed in a fashion similar to Section~\ref{S:stabilizing}\,, using techniques in~\cite{walsh:psc-sphere}.  However, to avoid an unnecessary technical discussion, we fix a PSC metric on $\sss$ containing a torpedo region, and a metric $g_{p,q}$ on each $\xpq$ with such a region disjoint from the spheres used to make the diffeomorphisms $f_{S_i}$.  As explained in the introduction, \thmref{psc} then follows from \thmref{diffs}.


\subsection*{\bf Stabilizing diffeomorphisms}\label{S:stabilizing}

We now address the technical issue of well-definedness of stabilization of diffeomorphisms under connected sum. A basic theorem in differential topology~\cite[Chapter 8]{hirsch}, going back to Palais~\cite{palais:disk} states that any two orientation preserving embeddings of an $n\textup{-disk}$ $D$ in an oriented $n$-manifold $M$ are smoothly isotopic.  (We assume throughout that all manifolds are connected.) This implies that the connected sum of oriented manifolds is well-defined up to diffeomorphism.

Now, given two diffeomorphisms $f:M \to M$ and $f':M' \to M'$, each restricting to the identity on embedded disks $D$ and $D'$, we can form the connected sum 
$$
f \cs f': M\cs M' \to M\cs M'
$$
as follows.  In the interior of each of $D$ and $D'$, fix smaller disks $\halfD$ and $\halfD'$, and specify 
$$
M\cs M' \ = \ (M - \halfD) \ \cup_{\partial}\  (M'-\halfD')
$$
with $f\cs f'$ defined to restrict to $f$ and $f'$ respectively on the two halves.  Note that $f\cs f'$ depends, up to isotopy, only on $f$ and $f'$ considered up to isotopy relative to $\halfD$ and $\halfD'$.

Suppose now that $f$ is an arbitrary orientation-preserving diffeomorphism and $f'= \idmprime$. By Palais' theorem, $f$ can be isotoped to a diffeomorphism $g$ that is the identity on a disk, and it might seem that we can define $f \cs f'$ to be $g \cs f'$.  We would like to say that this is a well-defined operation on isotopy classes; in other words, that it does not depend on the choice of isotopy between $f$ and $g$.  This is a well-known issue in the study of the mapping class group of surfaces, where there are stabilization results for mapping class groups of diffeomorphisms that are the identity on the boundary~\cite{harer:stability}, but not for the mapping class groups of closed surfaces.

For example, take $M$ and $M'$ to be genus $2$ surfaces, with $f = \idm$ and the isotopy from $f$ to $g$ to be one that drags a disk around an essential closed loop $\gamma \subset M$.  Then $f \cs f'$ would be represented by a composition of Dehn twists (with opposite senses) around a pair of homologous but non-isotopic curves, and hence would not be isotopic to the identity.  Similarly, one might vary $f$ by an isotopy from the identity to itself that rotates the disk $D$ by a non-trivial element of $\pi_1(SO(2))$.  Again, standard examples show that this operation can change the connected sum of diffeomorphisms by a Dehn twist around a separating curve on a surface.

In contrast, we will show that for simply-connected $4$-manifolds, there is a well-defined (up to isotopy) assignment $(f,f') \mapsto f\cs f'$ when $M'$ is $\pm\cptwo$ or $\sss$ and $f'$ is the identity diffeomorphism.  First, we collect some standard results about diffeomorphism groups, using the following notation.  For the moment $M$ will denote an oriented closed smooth $n$-manifold with a fixed embedding of an $n$-disk $D$. Then we have the following spaces, each with the $\calc^\infty$ topology:
\begin{itemize}
\smallskip
\item $\diff(M)$ is the group of orientation-preserving diffeomorphisms of $M$
\smallskip
\item $\diff_D(M)$ consists of diffeomorphisms whose restriction to $D$ is the identity
\smallskip
\item $\emb(D,M)$ is the space of orientation preserving embeddings of $D$ in $M$
\end{itemize}
\smallskip
Finally, $\calf(M)$ will denote the bundle of oriented (not necessarily orthonormal) frames in the tangent bundle.  As such, it is a fiber bundle over $M$ with fiber $GL_n^+(\br)$
The following facts are well-known to experts; the first is proved using parameterized versions of the isotopy extension and tubular neighborhood
theorems~\cite{cerf:embeddings,hirsch}, while the second is a parameterized version of the Cerf-Palais disk embedding theorem~\cite{cerf:embeddings,palais:disk}

\begin{lemma}\label{L:diffs} With the notation above,
\\[3pt]
\indent {\bf a)} Restricting a diffeomorphism to $D$ defines a fiber bundle $\diff(M) \to \emb(D,M)$ 
\\[1pt]
\indent \quad\  with fiber $\diff_D(M)$
\\[3pt]
\indent {\bf b)} The differential at $0 \in D$ of an embedding defines a homotopy equivalence 
\\[1pt] 
\indent \quad\ $\emb(D,M) \to \calf(M)$.
\end{lemma}

\vskip-10pt

An element in $\pi_1(SO(n))$, viewed as a smooth path beginning and ending at the identity, produces a diffeomorphism of $S^{n-1} \times I$.  If $S^{n-1} \times I$ is embedded in $M$, then this diffeomorphism extends naturally to $M$; we refer to such a diffeomorphism as a {\em Dehn twist}.  

\vskip-10pt

\begin{proposition}\label{P:independent} Let $M$ be a simply-connected $n$-manifold, and fix disks $\halfD \subset D\subset M$. Let $f:M\to M$ be an orientation-preserving diffeomorphism.   Then there a diffeomorphism $g: M\to M$ isotopic to $f$ that is the identity on $D$.  Any two choices of $g$ differ up to isotopy relative to $\halfD$ by composition of Dehn twists supported in $D - \textup{int}(\halfD)$.
\end{proposition}

The composition could mean no Dehn twists, in which case the two choices are isotopic.  We remind the reader that for $n\geq 3$, since $\pi_1(SO(n)) \cong \bz_2$, the composition of such a Dehn twist with itself is isotopic to the identity.

\vskip-10pt
\vskip-10pt

\begin{proof}
The existence of such a diffeomorphism $g$ follows, as mentioned above, from an isotopy of $D$ to $f(D)$.  Two choices of $g$ defined by such isotopies differ (up to isotopy of $M$) by an element of $\pi_1(\diff(M))$, which is of course an isotopy of $\idm$ to itself.  Combine the two parts of Lemma~\ref{L:diffs} and consider the homotopy long exact sequence of the fibration:
$$
\xymatrix{
 \pi_1(\diff_{\halfD}(M)) \ar[r] & \pi_1(\diff(M)) \ar[r] 
&  \pi_1(\emb(\halfD,M)) \ar[d]^\cong  \ar[r] & \pi_0(\diff_{\halfD}(M))  \\
&&  \pi_1(\calf(M)) &
}
$$

From this we see  that the obstruction to lifting this to an isotopy of $\idm$ to itself, relative to $\halfD$, comes from an element of $\pi_1(\calf(M))$.  But since $M$ is simply-connected, there is a surjection $\pi_1(SO(n)) \to \pi_1(\calf(M))$.  This implies the last statement of the proposition.
\end{proof}

\begin{theorem}\label{T:well-def}
Let $f:M\to M$ be an orientation-preserving diffeomorphism of a simply-connected $4$-manifold.  Then for $X = \pm\cptwo$ or $\sss$, there is a well-defined stabilization $f \cs \idx$ that depends up to isotopy only on the isotopy class of $f$.
\end{theorem}

\vskip-10pt
\vskip-10pt
\vskip-10pt

\begin{proof}
As suggested above, we isotope $f$ to a diffeomorphism that is the identity on an embedded disk $D$, and use it to define $f \cs \idx$.  By Proposition~\ref{P:independent} the resulting diffeomorphism of $M\cs X$ is well-defined, possibly up to Dehn twists supported on an $S^3 \times I$ region separating the two summands. We claim (compare~\cite[Theorem 2.4]{giansiracusa:mcg}) that for $X$ as indicated in the theorem, such a Dehn twist is isotopic to the identity.  To see this, consider an $S^1$ action on $X$ with a fixed point $p$; assume for convenience that it acts isometrically.   The local representation of $S^1$ on $T_pX$ defines a map $S^1 \to SO(4)$.  In the case of $\pm\cptwo$ and $\sss$, there is an action such that this map is the generator of $\pi_1(SO(4))$.  Using this action, it is easy to find an isotopy, supported on $X$, between a Dehn twist and the identity.
\end{proof}

\thmref{well-def} can be rephrased in terms of the group $\diff(M)$ of (orientation-preserving) diffeomorphisms of $M$. 

\begin{corollary}\label{C:pi0}
Let $M^4$ be simply-connected.  Stabilization with the identity diffeomorphism of $X= \pm\cptwo$ or $\sss$ gives a well-defined homomorphism 
$$
\Phi: \pi_0(\diff(M)) \to \pi_0(\diff(M \# X)).
$$
For $M = \xpq$ as in \thmref{diffs} and $X = \sss$, the kernel of $\Phi$ is infinite.
\end{corollary}
\begin{proof}
That $\Phi$ is well-defined is the content of \thmref{well-def}.  Recall that the multiplication in $\pi_0(G)$ for any topological group $G$ is induced from the multiplication in $G$.  Now given $f_0$ and $g_0$ in $\diff(M)$, we choose isotopies $f_t,\, g_t$ with $f_1,\, g_1$ the identity on an embedded disk $D$. Then $f_t \circ g_t$ gives an isotopy of $f_0 \circ g_0$ to $f_1 \circ g_1$, which is evidently also the identity on  $D$.  Since 
$$
(f_1 \circ g_1) \cs \rm{id}_{X} = \left(f_1  \cs \rm{id}_{X}\right) \circ \left(g_1\cs \rm{id}_{X}\right)
$$
we see that $\Phi([f_0] [g_0]) = \Phi([f_1] [g_1]) =  \Phi([f_1]) \Phi([g_1])$, so that $\Phi$ is a homomorphism.  Using the same notation as in the proof of \thmref{diffs}, we see that for $i \geq 2$, the diffeomorphisms $f_{S_i} \circ (f_{S_1})^{-1}$ form an infinite subset in the kernel of $\Phi$.  
\end{proof}

\affiliationone{
   Dave Auckly \\
   Department of Mathematics \\
   Kansas State University \\  
   Manhattan, KS 66506 \ USA \\
   \email{dav@math.ksu.edu}}
\affiliationtwo{
   Hee Jung Kim \\
   Department of Mathematical Sciences \\
   Seoul National University \\  
   Seoul, South Korea \\
   \email{heejungorama@gmail.com}}   
\affiliationthree{
   Paul Melvin \\
   Department of Mathematics \\
   Bryn Mawr College \\  
   Bryn Mawr, PA 19010 \ USA \\
   \email{pmelvin@brynmawr.edu}}   
\affiliationfour{
   Daniel Ruberman \\
   Department of Mathematics \\
   Brandeis University \\  
   Waltham, MA 02454 \ USA \\
   \email{ruberman@brandeis.edu}}   

\end{document}